\documentclass{amsart}

\usepackage{geometry}

\usepackage[T1]{fontenc}
\usepackage[utf8]{inputenc}

\usepackage{color}
\usepackage[colorlinks=true]{hyperref}
\hypersetup{urlcolor=blue, citecolor=blue}

\usepackage{amsmath,amssymb,amsfonts,amsthm} 
\usepackage{graphicx}
\usepackage{pgfplots}
\usepackage{cleveref}
\usepackage{thmtools}
\usepackage{stmaryrd}
\usepackage{mathtools}
\usepackage{mathrsfs}  
\usepackage{enumerate}
\usepackage{dsfont}

\usepackage[normalem]{ulem} 
\usepackage{cancel}
\usepackage[color=red!65!yellow]{todonotes}

\usepackage{subfig}
\usepackage{booktabs}

\newtheorem{thm}{Theorem}[section]
\newtheorem{lem}[thm]{Lemma}
\newtheorem{prop}[thm]{Proposition}

\newcounter{cst}
\newcommand{\ctel}[1]{C_{\refstepcounter{cst}\label{#1}\thecst}}
\newcommand{\cter}[1]{C_{\ref{#1}}}

\def\ds{\displaystyle}

\def\R{\mathbb{R}}

\def\be{\begin{equation}}
\def\ee{\end{equation}}

\def\n{{\boldsymbol n}}
\def\p{\partial}
\def\grad{{\nabla}}
\def\div{\grad\cdot}

\def\O{\Omega}

\def\Sig{\Sigma}
\def\sig{\sigma}

\def\k{\kappa}

\def\eps{\epsilon}
\def\x{{\boldsymbol x}}
\def\y{{\boldsymbol y}}
\def\d{{\rm d}}
\def\ov#1{\overline{#1}}

\def\wt#1{\widetilde{#1}}

\def\bphi{{\boldsymbol{\phi}}}

\def\bA{{\boldsymbol  A}}

\def\bF{{\boldsymbol  F}}
\def\bG{{\boldsymbol  G}}
\def\bH{{\boldsymbol  H}}

\def\bJ{{\boldsymbol  J}}

\def\bM{{\boldsymbol  M}}

\def\bS{{\boldsymbol  S}}

\def\bm{{\boldsymbol m}}
\def\bu{{\boldsymbol  u}}

\def\bv{{\boldsymbol v}}

\def\bh{{\boldsymbol h}}
\def\bmu{{\boldsymbol \mu}}
\def\bxi{{\boldsymbol \xi}}
\def\0{{\bf 0}}
\def\1{{\bf 1}}
\def\bGg{{\boldsymbol{\mathcal{G}}}}

\def\bbF{{\mathbb F}}

\def\bbP{{\mathbb P}}

\def\Ee{\mathcal{E}}

\def\Gg{\mathcal{G}}

\def\Kk{\mathcal{K}}

\def\Nn{\mathcal{N}}

\def\Pp{\mathcal{P}}

\def\Tt{\mathcal{T}}

\def\Zz{\mathcal{Z}}

\def\brho{{\boldsymbol\rho}}
\def\bpi{{\boldsymbol\pi}}

\def\dt{{\Delta t}}
\def\bff{{\boldsymbol f}}

\def\rhoe{\varrho}

\begin{document}
\title{A variational finite volume scheme for Wasserstein gradient flows}
\author[C. Cancès]{Clément Cancès}
\address{Cl\'ement Canc\`es (\href{mailto:clement.cances@inria.fr}{\tt clement.cances@inria.fr}):  Inria, Univ. Lille, CNRS, 
UMR 8524 - Laboratoire Paul Painlev\'e, F-59000 Lille 
} 
\author[T. O. Gallou\"et]{Thomas O. Gallou\"et}
\address{Thomas O. Gallou\"et:  (\href{mailto:thomas.gallouet@inria.fr}{\tt thomas.gallouet@inria.fr}):  INRIA, Project team Mokaplan, Universit\'e Paris-Dauphine, PSL Research University, UMR CNRS 7534-Ceremade
} 
\author[G. Todeschi]{Gabriele Todeschi}
\address{Gabriele Todeschi:  (\href{mailto:todeschi@ceremade.dauphine.fr}{\tt thomas.gallouet@inria.fr}):  INRIA, Project team Mokaplan, Universit\'e Paris-Dauphine, PSL Research University, UMR CNRS 7534-Ceremade
} 

\begin{abstract}
We propose a variational finite volume scheme to approximate the solutions to Wasserstein gradient flows. 
The time discretization is based on an implicit linearization of the Wasserstein distance expressed thanks 
to Benamou-Brenier formula, whereas space discretization relies on upstream mobility two-point flux approximation 
finite volumes. Our scheme is based on a first discretize then optimize approach in order to preserve the 
variational structure of the continuous model at the discrete level. Our scheme can be applied to a wide range of energies, 
guarantees non-negativity of the discrete solutions as well as decay of the energy. 
We show that our scheme admits a unique solution whatever the convex energy involved in the continuous problem, 
and we prove its convergence in the case of the linear Fokker-Planck equation with positive initial density. 
Numerical illustrations show that it is first order accurate in both time and space, and 
robust with respect to both the energy and the initial profile. 
\end{abstract}
\maketitle


\section{A strategy to approximate Wasserstein gradient flows}\label{sec:intro}

\subsection{Generalities about Wasserstein gradient flows}\label{ssec:GF}

Given a convex and bounded open subset $\O$ of $\R^d$, a strictly convex and proper energy functional $\mathcal{E}: L^1(\O;\R_+) \to [0,+\infty]$,
and given an initial density $\rho^0 \in L^1(\O;\R_+)$ with finite energy, i.e. such that $\Ee(\rho^0) < +\infty$, we want to solve problems of the form:
\be\label{eq:pde}
\begin{cases}
\p_t \rhoe - \div(\rhoe \grad \frac{\delta\Ee}{\delta\rho}[\rhoe]) = 0 & \text{in } Q_T = \O\times (0,T), \\
\rhoe \grad \frac{\delta\Ee}{\delta\rho}[\rhoe] \cdot \n = 0  & \text{on } \Sigma_T = \partial \O \times (0,T), \\
 \rhoe(\cdot,0) = \rho^0 & \text{in}\; \O.
 \end{cases}
\ee
\Cref{eq:pde} expresses the continuity equation for a time evolving density $\rhoe$, starting from the initial condition $\rho^0$, convected by the velocity field $-\nabla \frac{\delta \mathcal{E}}{\delta \rho}[\rhoe]$. The mixed boundary condition the system is subjected to represents a no flux condition across the boundary of the domain for the mass: the total mass is therefore preserved.

It is now well understood since the pioneering works of Otto~\cite{JKO98, Otto98, Otto01} that equations of the form of \eqref{eq:pde} can be interpreted as the 
gradient flow in the Wasserstein space w.r.t. the energy $\mathcal{E}$ \cite{AGS08}. A gradient flow is an evolution stemming from an initial condition and evolving at each time following the steepest decreasing direction of a prescribed functional.
Consider the space $\mathbb{P}(\O)$ of nonnegative measures defined on the bounded and convex domain $\Omega$ with prescribed total mass that are absolutely continuous w.r.t. the Lebesgue measure (hence $\mathbb{P}(\O) \subset L^1(\O;\R_+)$). 
The Wasserstein distance $W_2$ between two densities $\rho,\mu \in \mathbb{P}(\O)$ is the cost to transport one into the other in an optimal way with respect to the cost given by the squared euclidean distance, namely the optimization problem
\begin{equation}\label{W2}
W_2^2(\rho,\mu) = \displaystyle \min_{\gamma \in \Gamma(\rho,\mu)} \iint_{\Omega\times \Omega} |\y-\x|^2 d\gamma(\x,\y),
\end{equation}
with the set $\Gamma(\rho,\mu)$ of admissible transport plans given by
\[
\Gamma(\rho,\mu) = \Big\{ \gamma \in \mathbb{P}(\O\times\O) : \gamma^1 = \rho, \gamma^2=\mu \Big\},
\]
where $\gamma^1, \gamma^2$ denote the first and second marginal measure, respectively. 

A typical example of problem entering the framework of~\eqref{eq:pde} is the 
linear Fokker-Planck equation
\begin{equation}\label{eq:FokkerPlanck}
\partial_t \rhoe = \Delta \rhoe + \nabla \cdot (\rhoe \nabla V) \quad \text{in}\; Q_T,
\end{equation}
complemented with no-flux boundary conditions and an initial condition. In~\eqref{eq:FokkerPlanck}, $V\in W^{1,\infty}(\O)$ 
denotes a Lipschitz continuous exterior potential. 
In this case, the energy functional is 
\be\label{eq:NRJ_FP}
\mathcal{E}(\rho) = \int_{\Omega} [\rho \log \frac{\rho}{e^{-V}} - \rho + e^{-V}] \d \x.
\ee
The potential $V$ is defined up to an additive constant, which can be adjusted so that the 
densities $e^{-V}$ and $\rho^0$ have the same mass. 
Beside this simple example studied for instance in~\cite{JKO98, BGG12}, many problems have been proven to exhibit the same variational structure. 
Porous media flows \cite{Otto01, LM13_Muskat, CGM17}, magnetic fluids~\cite{Otto98}, supraconductivity~\cite{AS08, AMS11}, crowd motions~\cite{MRS10}, 
aggregation processes in biology~\cite{CDFLS11, Bla14_KS}, 
semiconductor devices modelling~\cite{KMX17}, or multiphase mixtures~\cite{CMN19, JKM_arXiv} are just few examples of problems that can be represented 
as gradient flows in the Wasserstein space. Designing efficient numerical schemes for approximating their solutions is therefore a major issue and our leading motivation.

\subsection{JKO semi-discretization}\label{ssec:JKO}

An intriguing question is how to solve numerically a gradient flow. Problem \eqref{eq:pde} can of course be directly 
discretized and solved using one of the many tools available nowadays for the numerical approximation of partial differential equations. 
The development of energy diminishing numerical methods based on classical ODE solvers for the march in time has been the purpose of many 
contributions in the recent past, see for instance~\cite{BC12, CG16_MCOM, CG_VAGNL, Cances_OGST, SCS18, ABPP_arxiv, CNV_HAL}. 
Nevertheless, the aforementioned methods disregard the fact that the trajectory aims at optimizing the energy decay, in opposition 
to methods based on minimizing movement scheme (often called JKO scheme after~\cite{JKO98}). 
This scheme can be thought as a generalization to the space $\bbP(\O)$  (the mass being defined by the initial data $\rho^0$) 
equipped with the metric $W_2$ of the backward Euler scheme and writes:
\begin{equation}\label{JKO}
\begin{cases}
 \rho_{\tau}^0 = \rho^0, \\
 \rho_{\tau}^n \in {\text{argmin}}_{\rho}\; \frac{1}{2\tau} W_2^2(\rho,\rho_{\tau}^{n-1}) + \mathcal{E}(\rho).
 \end{cases}
\end{equation}
The parameter $\tau$ is the time discretization step.
Scheme \eqref{JKO} generates a sequence of measures $\left(\rho_{\tau}^n\right)_{n\geq 1}$. 
Using this sequence it is possible to construct a time dependent measure by gluing them together in a piecewise constant (in time) fashion: $\rho_{\tau}(t) = \rho_{\tau}^n$, for $t \in (t^{n-1} = (n-1)\tau, t^n = n\tau]$. Under suitable assumptions on the functional $\mathcal{E}$, it is possible to prove the uniform convergence in time 
of this measure to weak solutions $\rhoe$ of \eqref{eq:pde} (see for instance~\cite{AGS08} or \cite{Santambrogio_OTAM}).

Lagrangian numerical methods appear to be very natural (especially in dimension 1) to approximate the Wasserstein distance and thus 
the solution to~\eqref{JKO}. This was already noticed in~\cite{KW99}, and motivated numerous contributions, see for instance~\cite{MO14, CG16, MO17, JMO17, 
CDMM18,CCP19,LMSS_arXiv}. In our approach, we rather consider an Eulerian method based on Finite Volumes for the space discretization. 
The link between monotone Finite Volumes and optimal transportation was simultaneously highlighted by  Mielke~\cite{Mie11} and 
Maas~\cite{Maas11, GM13, EM14, MM16, GKM_arXiv}. 
But these works only focuses on the space discretization, whereas we are interested in the fully discrete setting. Moreover, 
the approximation based on upstream mobility we propose in Section~\ref{ssec:scheme} does not enter their framework. 
Last but not least, let us mention the so-called ALG2-JKO scheme~\cite{BCL16, CGLM_preprint} where the optimization problem~\eqref{JKO}
is discretized and then solved thanks to an augmented Lagrangian iterative method. Our approach is close to the one of~\cite{BCL16}, 
with the goal to obtain a faster numerical solver. 
	
Thanks to formal calculations, let us highlight the connection of the minimization problem involved at each step of \eqref{JKO} with a 
system coupling a forward in time conservation law with a backward in time Hamilton-Jacobi (HJ) equation. 
The problem can be rewritten thanks to Benamou-Brenier dynamic formulation of optimal transport~\cite{BB00} as
\begin{equation}\label{JKO-BB}
\inf_{\rho, \bv} \frac{1}{2} \int_{t^{n-1}}^{t^n} \int_{\Omega} \rho |\bv|^2 \d \x \d t + \mathcal{E}(\rho(t^n)),
\end{equation}
where the density and velocity curves satisfy weakly
\begin{equation}\label{eq:continuity}
\begin{cases}
\partial_t \rho + \div (\rho \bv) = 0 & \text{in} \; \Omega \times (t^{n-1},t^n), \\
\rho \bv \cdot \n = 0 & \text{on } \; \partial \Omega \times (t^{n-1},t^n), \\
\rho(t^{n-1}) = \rho_\tau^{n-1}&\text{in}\;\Omega.
\end{cases}
\end{equation}
The next value $\rho_\tau^n$ is chosen equal to $\rho(t^n)$ for the optimal $\rho$ in~\eqref{JKO-BB}--\eqref{eq:continuity}. 
Using the momentum $\bm = \rho\bv$ instead of $\bv$ as a variable, and incorporating the constraint~\eqref{eq:continuity} in~\eqref{JKO-BB} 
yields the saddle-point problem 
\begin{multline}\label{JKO-primal}
\inf_{\rho, \bm } \sup_{\phi} \int_{t^{n-1}}^{t^n} \int_{\Omega} \frac{|\bm|^2}{2 \rho} \d \x\d t + 
\int_{t^{n-1}}^{t^n} \int_{\O} (\rho\partial_t \phi  + \bm \cdot \grad \phi) \d \x \d t \\
+ \int_{\Omega} [\phi(t^{n-1}) \rho_\tau^{n-1} - \phi(t^n) \rho(t^n)] \d \x + \mathcal{E}(\rho(t^n)).
\end{multline}
We will refer to \eqref{JKO-primal} as the primal problem.
The dual problem is obtained by exchanging inf and sup in~\eqref{JKO-primal}. 
Strong duality can be proven and the problem hence does not change. 
Optimizing first w.r.t. $\bm$ leads to $\bm =- \rho \grad \phi$, so that the dual problem writes
\be\label{JKO-dual}
\sup_{\phi} \inf_{\rho} \int_{t^{n-1}}^{t^n} \int_{\Omega} (\partial_t \phi -\frac{1}{2} |\grad \phi|^2) \rho \d \x \d t 
+ \int_{\Omega} [\phi(t^{n-1}) \rho_\tau^{n-1} -\phi(t^{n}) \rho(t^{n})]  \d \x  + \mathcal{E}(\rho(t^n)).
\ee
Because of the first term in~\eqref{JKO-dual}, the infimum is equal to $-\infty$ unless 
$-\partial_t \phi + \frac{1}{2} |\grad \phi|^2 \le 0$ a.e. in $\O\times(t^{n-1},t^n)$, with equality $\rho$-almost everywhere 
since $\rho \geq 0$. 
Moreover, optimizing w.r.t. $\rho(t^n)$ provides that $\phi(t^n) \leq \frac{\delta \Ee}{\delta \rho}[\rho(t^n)]$ with equality 
$\rho(t^n)$-almost everywhere. 
Hence the dual problem can be rewritten as
\be\label{eq:JKO-dual2}
\sup_{\phi(t^{n-1})}  \int_{\Omega} \phi(t^{n-1}) \rho_\tau^{n-1} \d \x + \inf_{\rho(t^n)} \left[\mathcal{E}(\rho(t^n)) - \int_{\Omega} \phi(t^n) \rho(t^n) \d \x\right],
\ee
subject to the constraints
\begin{equation}\label{eq:JKO-dual-oc}
\begin{cases}
-\partial_t \phi + \frac{1}{2} |\grad \phi|^2 \le 0 & \text{in}\; \O\times(t^{n-1},t^n), \\
\phi(t^n) \le \frac{\delta \mathcal{E}}{\delta \rho}[\rho(t^n)] & \text{in}\; \O. 
\end{cases}
\end{equation}

On the one hand the monotonicity of the backward HJ equation $- \p_t \phi + \frac{1}{2} |\grad \phi|^2 = f$ with respect to its right-hand side $f\leq 0$ implies that given $\phi(t^n)$, the solution (which exists) of $- \p_t \phi + \frac{1}{2} |\grad \phi|^2 = 0$ gives a bigger value at $\phi(t^{n-1})$ and thus a better competitor for \eqref{eq:JKO-dual-oc}. On the other given $\phi$, among all the $\rho$ satisfying the constraint, the competitor  $\bar{\rho} = \frac{\delta \mathcal{E}}{\delta \rho}^{-1} (\phi(t^n))$ is optimal. Indeed let $G_{\phi(t^n)} (\rho) = \mathcal{E}(\rho)-\int_{\Omega} \phi(t^n) \rho$, the convexity of $\mathcal{E}$ at $ \bar{\rho}$ implies
\[
\begin{aligned}
 G_{\phi(t^n)}(\rho)= \mathcal{E}({\rho}) - \int_{\Omega} \phi(t^n) \rho \d \x  &\ge \int_{\Omega} \frac{\delta \mathcal{E}(\bar{\rho})}{\delta \rho} (\rho - \bar{\rho}) \d \x + \mathcal{E}(\bar{\rho}) - \int_{\Omega} \phi(t^n) \rho \d \x  \\
&= \int_{\Omega} \phi(t^n) (\rho-\bar{\rho}) \d \x +  \mathcal{E}(\bar{\rho})  -\int_{\Omega} \phi(t^n)  \bar{\rho} \d \x =G_{\phi(t^n)}(\bar{\rho}).
\end{aligned}
\]

At the end of the day, the primal-dual optimality conditions of problem \eqref{JKO} finally amounts to the mean field game
\be\label{eq:MFG}
\begin{cases}
\p_t \phi - \frac{1}{2} |\grad \phi|^2 =  0, \\
\p_t \rho - \div(\rho \grad \phi) = 0, 
\end{cases}
\text{in}\; \O\times(t^{n-1},t^n), \; \text{with}\;
\qquad
\begin{cases}
\rho(t^{n-1}) = \rho_\tau^{n-1}, \\
\phi(t^n) = \frac{\delta \Ee}{\delta \rho}[\rho(t^n)],
\end{cases}
\text{in}\; \O.
\ee
The optimal $\rho_\tau^n$ of~\eqref{JKO} is then equal to $\rho(t^n)$.
The no-flux boundary condition reduces to $\grad \phi \cdot \n = 0$  on $\p\O\times(t^{n-1},t^n)$.

The approximation of the system~\eqref{eq:MFG} is a natural strategy to approximate the solution to~\eqref{eq:pde}. 
This approach was for instance at the basis of the works~\cite{BCL16,CCWW_ArXiv}. These methods require a sub-time stepping 
to solve system~\eqref{eq:MFG} on each interval $(t^{n-1},t^n)$, yielding a possibly important computational cost. 
The avoidance of this sub-time stepping is the main motivation of the time discretization we propose now.

\subsection{Implicit linearization of the Wasserstein distance and LJKO scheme}\label{ssec:LJKO}
Let us introduce in the semi-discrete in time setting the time discretization to be used in the fully discrete setting later on. 
The following ansatz is at the basis of our approach: when $\tau$ is small, $\rho_\tau^n$ is close to $\rho_\tau^{n-1}$. 
Then owing to~\cite[Section 7.6]{Villani-Topics} (see also~\cite{Peyre18}), the Wasserstein distance between two densities $\rho$ and $\mu$ of $\bbP(\O)$
is close to some weighted $H^{-1}$ distance, namely
\be\label{eq:W2_H-1}
\left\| \rho - \mu \right\|_{\dot H_\rho^{-1}} = W_2(\rho, \mu) + o(W_2(\rho, \mu)), \qquad \forall \rho, \mu \in \bbP(\O). 
\ee
In the above formula, we denoted by 
\be\label{eq:H-1}
\left\| h\right\|_{\dot{H}^{-1}_{\rho}} = \left\{ \sup_{\varphi} \int_{\Omega} h \varphi \, \d\x \;\middle|\;  {\|\varphi\|}_{\dot{H}^1_{\rho}} \le 1 \right\}, 
\quad\text{with}\; {\|\varphi\|}_{\dot{H}^1_{\rho}} = \left(\int_\O \rho |\grad \varphi|^2 \d\x\right)^{1/2}, 
\ee
so that $\| \rho - \mu\|_{\dot{H}^{-1}_{\rho}} = \|\psi\|_{\dot{H}^{1}_{\rho}}$ with $\psi$ solution to
\begin{equation}\label{eq:monge-ampere-lin}
\begin{cases} \rho - \mu- \div (\rho \grad \psi) = 0 & \text{in}\; \O,\\
\grad \psi \cdot \n = 0 & \text{on}\; \p\O. 
\end{cases}\end{equation}
Indeed, in view of~\eqref{eq:H-1}--\eqref{eq:monge-ampere-lin}, there holds
\[
\int_{\Omega} (\rho-\mu) \varphi \, \d\x = - \int_{\Omega} \div (\rho \grad \psi) \varphi \, \d \x 
= \int_{\Omega} \rho \grad \psi \cdot \grad \varphi \, \d\x \leq \|\psi\|_{\dot{H}^1_\rho} \|\varphi\|_{\dot{H}^1_\rho},
\]
with equality if $\varphi = \psi /\|\psi\|_{\dot{H}^1_\rho}$. Equation~\eqref{eq:monge-ampere-lin} can be thought as 
a linearization of the Monge-Amp\`ere equation. 

In view of \eqref{eq:W2_H-1}, a natural idea is to replace the Wasserstein distance by the weighted 
$\dot{H}_\rho^{-1}$ norm in~\eqref{JKO}, leading to what we call the implicitly linearized JKO (or LJKO) scheme: 
\begin{equation}\label{LJKO}
\rho_{\tau}^n \in \underset{\rho \in \bbP(\O)}{\text{argmin}} \frac{1}{2\tau} \left\|\rho - \rho_\tau^{n-1}\right\|^2_{\dot{H}^{-1}_{\rho}(\Omega)} + \mathcal{E}(\rho), 
\qquad n \geq 1. 
\end{equation}
The choice of an implicit weight $\rho$ in~\eqref{LJKO} appears to be particularly important when $\{\rho^{n-1}_\tau=0\}$ has a non-empty interior set, 
which can not be properly invaded by the $\rho_\tau^n$ if one chooses the explicit (but computationally cheaper) weight $\rho_\tau^{n-1}$ as in~\cite{MW19}.
Our time discretization is close to the one that was proposed very recently in~\cite{LLW19} where the introduction on inner time stepping was also avoided. 
In \cite{LLW19}, the authors introduce a regularisation term based on Fisher information, which mainly amounts to stabilize the scheme 
thanks to some additional non-degenerate diffusion. In our approach, we manage to avoid this additional stabilization term by taking advantage of the 
monotonicity of the involved operators. 

At each step $n\geq 1$, \eqref{LJKO} can be formulated as a constrained optimization problem. To highlight its convexity, 
we perform the change of variables $(\rho,\psi) \mapsto (\rho,\bm = -\rho \grad \psi)$, in analogy with \eqref{JKO-BB}, and rewrite step $n$ as:
\begin{equation}\label{LJKO-dyn}
\inf_{\rho,\bm} \int_{\O} \frac{|\bm|^2}{2\tau \rho} \d \x + \mathcal{E}(\rho), \quad \text{subject to: }
\left\{
\begin{aligned}
&\rho - \rho_\tau^{n-1} + \div \bm = 0 &&\text{in } \, \Omega, \\
&\bm \cdot \n = 0 &&\text{on } \, \partial \Omega.
\end{aligned}
\right.
\end{equation}
Incorporating the constraint in the above formulation yields the following inf-sup problem:
\begin{equation}\label{LJKO-infsupm}
\inf_{\rho,\bm} \sup_{\phi} \int_{\O} \frac{|\bm|^2}{2\tau \rho} \d \x -\int_{\O} (\rho-\rho_\tau^{n-1}) \phi \,\d \x + \int_{\O} \bm \cdot
 \grad \phi \,\d \x + \mathcal{E}(\rho),
\end{equation}
the supremum w.r.t. $\phi$ being $+\infty$ unless the constraint is satisfied. Problem \eqref{LJKO-infsupm} is strictly convex in $(\rho,\bm)$ 
and concave (since linear) in $\phi$.
Exploiting Fenchel-Rockafellar duality theory it is possible to show that strong duality holds, so that \eqref{LJKO-infsupm} is equivalent to its dual 
problem where the inf and the sup have been swapped.
Optimizing w.r.t. to $\bm$ yields the optimality condition $\bm = -\tau \rho \grad \phi$, hence
the problem reduces to
\begin{equation}\label{LJKO-supinf}
\sup_{\phi}   \int_{\Omega} \rho_\tau^{n-1} \phi \,\d \x +\inf_{\rho} \int_{\O} (-\phi -\frac{\tau}{2} |\grad \phi|^2) \rho \,\d \x + \mathcal{E}(\rho).
\end{equation}
The problem is now strictly convex in $\rho$ and concave in $\phi$. Optimizing w.r.t. $\rho$ leads to the optimality condition
\be\label{eq:HJ_disc_leq}
\phi_\tau^n + \frac{\tau}{2} |\grad \phi_\tau^n|^2 \le \frac{\delta \mathcal{E}}{\delta \rho}[\rho_\tau^n], 
\ee
with equality on $\{\rho_\tau^n>0\}$. In the above formula, $\phi_\tau^n$ denote the optimal $\phi$ realizing the sup in~\eqref{LJKO-supinf}.
Similarly to was was done in the previous section for the JKO scheme, it is possible to show again that saturating inequality~\eqref{eq:HJ_disc_leq} 
on $\{\rho_\tau^n=0\}$ is optimal since the mapping $f \mapsto \phi$ solution to $\phi + \frac{\tau}2 |\grad \phi|^2 = f$ is monotone. 
Finally, the optimality conditions for the LJKO problem~\eqref{LJKO} write
\be\label{eq:LJKO-oc}
\begin{cases}
\displaystyle \phi_\tau^n + \frac{\tau}{2} |\grad \phi_\tau^n|^2 = \frac{\delta \mathcal{E}}{\delta \rho}[\rho_\tau^n], \\
\displaystyle\frac{\rho_\tau^n-\rho_\tau^{n-1}}{\tau} - \nabla \cdot (\rho_\tau^n \grad \phi_\tau^n) = 0, 
\end{cases}
\ee
set on $\O$, complemented with homogeneous Neumann boundary condition $\grad \phi_\tau^n \cdot \n = 0$ on $\p\O$.
We can interpret \eqref{eq:LJKO-oc} as the one step resolvent of the mean-field game \eqref{eq:MFG}. 
Both the forward in time continuity equation and the backward in time HJ equation are discretized thanks to one step of  
backward Euler scheme.

\subsection{Goal and organisation of the paper}\label{ssec:goal}

As already noted, most of the numerical methods based on backward Euler scheme disregard the optimal 
character of the trajectory $t \mapsto \rhoe(t)$ of the exact solution to~\eqref{eq:pde}. 
Rather than discretizing directly the PDE~\eqref{eq:pde}, which can be thought as the Euler-Lagrange equation 
for the steepest descent of the energy, we propose to first discretize w.r.t. space the functional appearing in the optimization problem~\eqref{LJKO}, 
and then to optimize. The corresponding Euler-Lagrange equations will then encode the optimality of the trajectory. 
The choice of the LJKO scheme~\eqref{LJKO} rather than the classical JKO scheme~\eqref{JKO} is motivated by the fact that 
solving~\eqref{eq:LJKO-oc} is computationally affordable. Indeed, it merely demands to approximate two functions $\rho_\tau^n, \phi_\tau^n$ 
rather than time depending trajectories in function space as for the JKO scheme~\eqref{eq:MFG}. This allows in particular to avoid 
inner time stepping as in~\cite{BCL16,CCWW_ArXiv}, making our approach much more tractable to solve complex problems. 

Two-Point Flux Approximation (TPFA) Finite Volumes are a natural solution for the space discretization. 
The are naturally locally conservative thus well-suited to approximate the conservation laws. Moreover, 
they naturally transpose to the discrete setting the monotonicity properties of the continuous operators. 
Monotonicity was crucial in the derivation of the optimality conditions~\eqref{eq:LJKO-oc}, as it will 
also be the case in the fully discrete framework later on. 
This led us to use upstream mobilities in the 
definition of the discrete counterpart of the squared $\dot H^1_\rho$ norm. 
The system~\eqref{eq:LJKO-oc} thus admits a discrete counterpart~\eqref{LJKOh}.
The derivation of the fully discrete Finite Volume scheme based on the LJKO time discretization is 
performed in Section~\ref{sec:FV}, where we also establish the well-posedness of the scheme, 
as well as the preservation at the discrete level of fundamental properties of the continuous model, 
namely the non-negativity of the densities and the decay of the energy along time. 
In Section~\ref{sec:conv}, we show that our scheme converges in the case of the Fokker-Planck equation~\eqref{eq:FokkerPlanck} 
under the assumption that the initial density is bounded from below by a positive constant. Even though 
we do not treat problem~\eqref{eq:pde} in its full generality, this results shows for the consistency 
of the scheme. Finally, Section~\ref{sec:num} is devoted to numerical results, where our scheme is tested 
on several problems including systems.


\section{A variational Finite Volume scheme}\label{sec:FV}

The goal of this section is to define the fully discrete scheme to solve~\eqref{eq:pde}, and to 
exhibit some important properties of the scheme. But at first, let us give some assumptions and 
notations on the mesh. 

\subsection{Discretization of $\O$}\label{ssec:mesh}

The domain $\O\subset \R^d$ is assumed to be polygonal if $d=2$ or polyhedral if $d=3$.
The specifications on the mesh are classical for TPFA Finite Volumes~\cite{EGH00}. 
More precisely, an \emph{admissible mesh of $\O$} is a triplet $\left(\Tt, \ov \Sigma, {(\x_K)}_{K\in\Tt}\right)$ such that the following conditions are fulfilled. 
\begin{enumerate}[(i)]
\item Each control volume (or cell) $K\in\Tt$ is non-empty, open, polyhedral and convex. We assume that 
$K \cap L = \emptyset$ if $K, L \in \Tt$ with $K \neq L$, while $\bigcup_{K\in\Tt}\ov K = \ov \O$. 
The Lebesgue measure of $K\in\Tt$ is denoted by $m_K>0$.
\item Each face $\sig \in \ov \Sigma$ is closed and is contained in a hyperplane of $\R^d$, with positive 
$(d-1)$-dimensional Hausdorff (or Lebesgue) measure denoted by $m_\sig = \mathcal{H}^{d-1}(\sig) >0$.
We assume that $\mathcal{H}^{d-1}(\sig \cap \sig') = 0$ for $\sig, \sig' \in \ov \Sigma$ unless $\sig' = \sig$.
For all $K \in \Tt$, we assume that 
there exists a subset $\ov \Sigma_K$ of $\ov \Sigma$ such that $\p K =  \bigcup_{\sig \in \ov \Sigma_K} \sig$. 
Moreover, we suppose that $\bigcup_{K\in\Tt} \ov \Sigma_K = \ov \Sigma$.
Given two distinct control volumes $K,L\in\Tt$, the intersection $\ov K \cap \ov L$ either reduces to a single face
$\sig  \in \ov \Sigma$ denoted by $K|L$, or its $(d-1)$-dimensional Hausdorff measure is $0$. 
\item The cell-centers $(\x_K)_{K\in\Tt} \subset \O$ are pairwise distinct 
and are such that, if $K, L \in \Tt$ 
share a face $K|L$, then the vector $\x_L-\x_K$ is orthogonal to $K|L$ and has the same orientation as 
the normal $\n_{KL}$ to $K|L$ outward w.r.t. $K$.
\end{enumerate} 
Cartesian grids, Delaunay triangulations or Vorono\"i tessellations are typical examples of admissible meshes in the above sense.
Since no boundary fluxes appear in our problem, the boundary faces $\Sigma_{\rm ext} = \{ \sig \subset \p\O\}$ are not involved in our computations. 
Nonzeros fluxes may only occur across internal faces $\sig \in \Sigma = \ov \Sigma \setminus \Sigma_{\rm ext}$. 
We denote by $\Sigma_{K} = \ov \Sigma_{K}\cap \Sigma$ the internal faces belonging to $\p K$, and by $\Nn_K$ the neighboring 
cells of $K$, i.e., $\Nn_K = \{ L \in \Tt \;|\; K|L \in \Sigma_{K}\}$.
For each internal face $\sigma = K|L \in \Sigma$, we refer to the diamond cell $\Delta_\sig$ as the polyhedron whose edges 
join $\x_K$ and $\x_L$ to the vertices of $\sigma$. The diamond cell $\Delta_\sig$ is convex if $\x_K \in K$ and $\x_L \in L$. 
Denoting by $d_\sig = |\x_K-\x_L|$, the measure $m_{\Delta_{\sigma}}$ of $\Delta_\sig$ is then equal to $m_\sig d_\sig/d$, where 
$d$ stands for the space dimension. The transmissivity of the face $\sig \in \Sigma$ is defined by $a_\sig = m_\sig / d_\sig$.

The space $\R^\Tt$ is equipped with the scalar product 
\[
\langle \bh, \bphi \rangle_\Tt = \sum_{K\in\Tt} h_K \phi_K m_K, \qquad \forall \bh = {(h_K)}_{K\in\Tt}, \bphi =  {(\phi_K)}_{K\in\Tt},
\]
which mimics the usual scalar product on $L^2(\O)$.

\subsection{Upstream weighted dissipation potentials}\label{ssec:H1rhoTt}

Since the LJKO time discretization presented in Section~\ref{ssec:LJKO} relies on weighted $\dot H^1_\rho$ and $H^{-1}_\rho$ norms, 
we introduce the discrete counterparts to be used in the sequel. As it will be explained in what follows, the upwinding 
yields to problems to introduce discrete counterparts to the norms. To bypass this difficulty, we adopt a formalism based 
on dissipation potentials inspired from the one of generalized gradient flows introduced by Mielke in~\cite{Mie11}.
This framework was used for instance to study the convergence of the semi-discrete in space squareroot Finite Volume approximation of the 
Fokker-Planck equation, see~\cite{Heida18}.

Let $\brho = \left(\rho_K\right)_{K\in\Tt} \in \R_+^\Tt$, and let $\bphi = \left(\phi_K\right)_{K\in\Tt} \in \R^\Tt$, 
then we define the upstream weighted discrete counterpart of $\frac12\|\phi\|_{\dot H^1_\rho}^2$ by 
\be\label{eq:H1rhoTt}
\Psi_{\Tt}^*(\brho; \bphi) 
=  \frac12\sum_{\substack{\sig \in \Sig\\\sig = K|L}} a_\sig \rho_\sig \left(\phi_K - \phi_L\right)^2 \geq 0, 
\ee
where $\rho_\sig$ denotes the upwind value of $\brho$ on $\sig \in \Sig$:
\be\label{eq:upwind-1}
\rho_\sig = \begin{cases}
\rho_K & \text{if}\; \phi_K > \phi_L, \\
\rho_L & \text{if}\; \phi_K < \phi_L, 
\end{cases}
\qquad \forall \sig = K|L \in \Sig. 
\ee
Because of the upwind choice of the mobility~\eqref{eq:upwind-1}, the functional~\eqref{eq:H1rhoTt} is not symmetric, 
i.e., $\Psi_\Tt^*(\brho;\bphi) \neq \Psi_\Tt^*(\brho;-\bphi)$ in general, which prohibits to define a semi-norm from $\Psi_{\Tt}^*(\brho;\cdot)$. 
But one easily checks that $\bphi \mapsto \Psi_\Tt^*(\brho,\bphi)$ is convex, continuous thus lower semi-continuous (l.s.c.) and proper. 

Let us now turn to the definition of the discrete counterpart of $\frac12\|\cdot\|^2_{\dot H^{-1}_\rho}$.
To this end, we introduce the space $\bbF_\Tt \subset \R^{2\Sig}$ of conservative fluxes. 
An element $\bF$ of $\bbF_\Tt$ is made of two outward fluxes $F_{K\sig}, F_{L\sig}$ for each $\sig = K|L \in \Sig$, 
and one flux $F_{K\sig}$ per boundary face $\sig \in \Sig_K$. 
We impose the conservativity across each internal faces
\be\label{eq:cons_sig}
F_{K\sig}+F_{L\sig} = 0, \qquad \forall \sig = K|L \in \Sigma.
\ee
In what follows, we denote by $F_\sig = |F_{K\sig}| = |F_{L,\sig}|$.  
There are no flux across the boundary faces.
The space $\bbF_\Tt$ is then defined as 
\[
\bbF_\Tt = \left\{ \bF = \left(F_{K\sig}, F_{L\sig}\right)_{\sig=K|L  \in \Sig} \in \R^{2\Sig}\; \middle| \; \text{\eqref{eq:cons_sig} holds} \right\}.
\]
Now, we define the subspace 
\[
\R_0^\Tt = \left\{ \bh = \left(h_K\right)_{K\in\Tt} \in \R^\Tt \; \middle| \; \langle \bh, \1 \rangle_\Tt = 0 \right\}
\]
and
\be\label{eq:H-1rhoTt}
\Psi_\Tt(\brho; \bh) =  \inf_{\bF} \sum_{\sig \in \Sig} \frac{(F_\sig)^2}{2 \rho_\sig} d_\sig m_\sig \geq 0, \qquad \forall \bh \in \R_0^\Tt,
\ee
where the minimization over $\bF$ is restricted to the linear subspace of $\bbF_\Tt$ such that 
\be\label{eq:cons_K}
h_Km_K = \sum_{\sig \in \Sigma_K}m_\sig F_{K\sig}, \qquad \forall K \in \Tt.
\ee
In~\eqref{eq:H-1rhoTt}, $\rho_\sig$ denotes the upwind value w.r.t. $\bF$, i.e., 
\be\label{eq:upwind}
\rho_\sig = \begin{cases}
\rho_K & \text{if}\;F_{K\sig} >0, \\
\rho_L & \text{if}\;F_{L\sig} >0, 
\end{cases}
\qquad \forall \sig = K|L \in \Sigma.
\ee
In the case where some $\rho_\sig$ vanish, we adopt the following convention in~\eqref{eq:H-1rhoTt} and in what follows:
\[
 \frac{(F_\sig)^2}{2\rho_\sig} = \begin{cases}
 0 & \text{if}\; F_\sig = 0 \; \text{and}\; \rho_\sig =0, \\
 +\infty & \text{if}\; F_\sig > 0 \; \text{and}\; \rho_\sig =0, 
 \end{cases}
 \qquad \forall \sig \in \Sig.
\]
Remark that this condition is similar to the one implicitly used in \eqref{JKO-primal} and \eqref{LJKO-dyn}. 
Summing~\eqref{eq:cons_K} over $K\in\Tt$ and using the conservativity across the edges~\eqref{eq:cons_sig}, one notices that 
there is no $\bF \in \bbF_\Tt$ satisfying~\eqref{eq:cons_K} unless $\bh \in \R^\Tt_0$. 
But when $\bh \in \R^\Tt_0$, the minimization set in~\eqref{eq:H-1rhoTt} is never empty. 
Note that $\Psi_\Tt(\brho;\bh)$ may take infinite values 
when $\brho$ vanishes on some cells, for instance $\Psi_\Tt(\brho;\bh)=+\infty$  if $h_K>0$ and $\rho_K=0$ for some $K\in\Tt$.

Formula \eqref{eq:H-1rhoTt} deserves some comments. This sum is built to approximate $\int_\O \frac{|\bm|^2}{2\rho} \d\x$.
The flux $F_\sig$ approximates $|\bm \cdot \n_\sig|$, and thus encodes the information on $\bm$ only in the one direction (normal to the face $\sig$) 
over $d$. But on the other hand, the volume $d_\sig m_\sig$ is equal to $d m_{\Delta_\sig}$ which allows to hope that the sum is a consistent 
approximation of the integral. This remark has a strong link with the notion of inflated gradients introduced in~\cite{CHLP03,EG03}. 
The convergence proof carried out in Section~\ref{sec:conv} somehow shows the non-obvious consistency of this formula. 

At the continuous level, the norms $\|\cdot\|_{\dot H^1_\rho}$ and $\|\cdot\|_{H^{-1}_\rho}$ are in duality.
This property is transposed to the discrete level in the following sense. 
\begin{lem}\label{lem:Legendre}
Given $\brho\geq\0$, the functionals $\bh\mapsto\Psi_\Tt(\brho;\bh)$ and $\bphi\mapsto \Psi^*_\Tt(\brho;\bphi)$ 
are one another Legendre transforms in the sense that 
\be\label{eq:Legendre}
\Psi_\Tt(\brho;\bh) = \sup_{\bphi} \langle \bh, \bphi \rangle_\Tt - \Psi_\Tt^*(\brho;\bphi), \qquad \forall \bh \in \R_0^\Tt.
\ee
In particular, both are proper convex l.s.c. functionals. Moreover, if $\Psi_\Tt(\brho;\bh)$ is finite, then 
there exists a discrete Kantorovitch potential $\bphi$ solving 
\be\label{eq:Kanto_disc}
h_K m_K  = \sum_{\substack{\sig \in \Sig_K\\\sig = K|L}} a_\sig \rho_\sig (\phi_K - \phi_L), \qquad \forall K \in \Tt, 
\ee
such that 
\be\label{eq:=Dissip}
\Psi_\Tt(\brho;\bh) =\Psi_\Tt^*(\brho;\bphi) = \frac12 \langle \bh, \bphi \rangle_\Tt.
\ee
\end{lem}
\begin{proof} Let $\brho \geq \0$ be fixed. 
Incorporating the constraint~\eqref{eq:cons_K} in~\eqref{eq:H-1rhoTt}, and using the definition of $\rho_\sig$ and the twice
conservativity constraint~\eqref{eq:cons_sig}, we obtain the saddle point primal problem
\begin{multline*}
\Psi_\Tt(\brho; \bh) = \inf_{\bF} \sup_{\bphi} \sum_{\substack{\sig \in \Sig\\\sig = K|L}} 
\left[\frac{\left((F_{K\sig})^+\right)^2}{2\rho_K}  + \frac{\left((F_{K\sig})^-\right)^2}{2\rho_L}\right]m_\sig d_\sig  \\
 + \sum_{K\in\Tt} h_K \phi_K m_K - \sum_{\substack{\sig \in \Sig\\\sig = K|L}} m_\sig F_{K\sig} (\phi_K - \phi_L).
\end{multline*}
The functional in the right-hand side is convex and coercive w.r.t. $\bF$ and linear w.r.t. $\bphi$, so that strong duality holds. 
We can exchange the sup and the inf in the above formula to obtain the dual problem, and we minimize first w.r.t. $\bF$, leading to 
\[
F_{K\sig}= \rho_\sig \frac{\phi_K - \phi_L}{d_\sig}, \qquad \forall \sig = K|L \in \Sig.
\]
Substituting $F_{K\sig}$ by $\rho_\sig \frac{\phi_K - \phi_L}{d_\sig}$ in the dual problem leads to~\eqref{eq:Legendre}, 
while the constraint~\eqref{eq:cons_K} turns to \eqref{eq:Kanto_disc}. The fact that $\Psi_\Tt^*(\brho,\cdot)$ is also the Legendre transform 
of $\Psi_\Tt(\brho,\cdot)$ follows from the fact that it is convex l.s.c., hence equal to its relaxation. 

When $\Psi_\Tt(\brho; \bh)$ is finite, then the supremum 
in~\eqref{eq:Legendre} is achieved, ensuring the existence of the corresponding discrete Kantorovitch potentials $\bphi$.
Finally,  multiplying~\eqref{eq:Kanto_disc} by 
the optimal $\phi_K$ and by summing over $K\in\Tt$ yields $\langle\bh, \bphi \rangle_\Tt = 2 \Psi^*_\Tt(\brho; \bphi)$. 
Substituting this relation in~\eqref{eq:Legendre} shows the relation $\Psi_\Tt(\brho;\bh) =\Psi_\Tt^*(\brho;\bphi)$.
\end{proof}

Our next lemma can be seen as an adaptation to our setting of a well known 
properties of optimal transportation, namely $\rho \mapsto \frac12W_2^2(\rho,\mu)$ is convex, which is key in the study of Wasserstein gradient flows.
\begin{lem}\label{lem:convex}
Let $\bmu \in \R^\Tt_+$, 
the function $\brho\mapsto \Psi_\Tt(\brho;\bmu-\brho)$ is proper and convex on 
$
(\bmu +\R_0^\Tt)\cap \R_+^\Tt.
$
\end{lem}
\begin{proof}
The function $\brho\mapsto \Psi_\Tt(\brho;\bmu-\brho)$ is proper since it is equal to $0$ at $\brho = \bmu$. 
Then it follows from~\eqref{eq:Legendre} that 
\be\label{eq:PsiTt}
\Psi_\Tt(\brho;\bmu-\brho) = \sup_{\bphi} \langle \bmu-\brho, \bphi \rangle_\Tt - \Psi_\Tt^*(\brho;\bphi).
\ee
Since  $\brho\mapsto \Psi_\Tt^*(\brho;\bphi)$ is linear, $\Psi_\Tt(\brho;\bmu-\brho)$ is defined as the supremum of linear functions, whence it is convex.
\end{proof}

\subsection{A variational upstream mobility Finite Volume scheme}\label{ssec:scheme}

The finite volume discretization replaces the functions $\rho_\tau^n, \phi_\tau^n$ at time step $n \geq 1$ defined on $\Omega$ with the 
vectors $\brho^n \in \R_+^\Tt$ and $\bphi^n \in \mathbb{R}^{\Tt}$. In each cell $K$, the restriction of each of these functions is 
approximated by a single real number $\rho_K^n, \phi_K^n$, which can be thought as its mean value located 
in the cell center $\x_K$. Given $\brho^0 \in \R_+^\Tt$, 
the space $\bbP_\Tt$ which is the discrete counterpart of $\bbP(\O)$ is then defined by 
\[
\bbP_\Tt = \left\{ \brho \in \R_+^\Tt \; \middle| \; \langle \brho, \1 \rangle_\Tt = \langle \brho^0, \1 \rangle_\Tt \right\}
=(\brho^0 + \R_0^\Tt) \cap \R_+^\Tt.
\]
It is compact. 
The energy $\Ee$ is discretized into a strictly convex functional $\Ee_\Tt \in C^1(\R_+^\Tt ;\R_+)$ that we do not 
specify yet. We refer to Sections~\ref{sec:conv} and~\ref{sec:num} for explicit examples.

We have introduced all the necessary material to introduce our numerical scheme, which combines
upstream weighted Finite Volumes for the space discretization and the LJKO time discretization:
\be\label{eq:LJKO_0}
\brho^n \in \underset{\brho \in \bbP_\Tt}{\text{argmin}}\frac1{\tau} \Psi_\Tt(\brho;\brho^{n-1}-\brho) + \Ee_\Tt(\brho), \qquad n \geq 1. 
\ee
A further characterization of the scheme is needed for its practical implementation, but the condensed expression~\eqref{eq:LJKO_0} 
already provides crucial informations gathered in the following theorem. Note in particular that our scheme 
automatically preserves mass and the positivity since the solutions $\left(\brho^n\right)_{n\geq 1}$ belong to $\bbP_\Tt$.
\begin{thm}\label{thm:LJKO}
For all $n \geq 1$, there exists a unique solution $\brho^n \in \bbP_\Tt$ to~\eqref{eq:LJKO_0}.
Moreover, energy is dissipated along the time steps. More precisely, 
\be\label{eq:NRJ_0}
\Ee_\Tt(\brho^n)\leq \Ee_\Tt(\brho^n) + \frac1{\tau} \Psi_\Tt(\brho^n;\brho^{n-1}-\brho^{n}) \leq \Ee_\Tt(\brho^{n-1}), \qquad \forall n \geq 1.
\ee
\end{thm}
\begin{proof} The functional $\brho \mapsto \frac1{\tau} \Psi_\Tt(\brho;\brho^{n-1}-\brho) + \Ee_\Tt(\brho)$ l.s.c. and strictly convex on 
the compact set $\bbP_\Tt$ in view of Lemma~\ref{lem:convex} and of the assumptions on $\Ee_\Tt$. 
Moreover, it is proper since $\brho^{n-1}$ belongs to its domain.
Therefore, it admits a unique minimum on $\bbP_\Tt$. The energy / energy dissipation estimate~\eqref{eq:NRJ_0} is 
obtained by choosing $\brho = \brho^{n-1}$ as a competitor in~\eqref{eq:LJKO_0}.
\end{proof}

In view of~\eqref{eq:PsiTt}, and after rescaling the dual variable $\bphi \leftarrow \frac\bphi\tau$, 
solving~\eqref{eq:LJKO_0} amounts to solve the saddle point problem 
\be\label{eq:LJKO_primal}
\inf_{\brho\geq \0} \sup_{\bphi} \left\langle \brho^{n-1} - \brho, \bphi \right\rangle_\Tt - 
\frac\tau2 \sum_{\substack{\sig \in \Sig\\\sig = K|L}} a_\sig \rho_\sig (\phi_K - \phi_L)^2 + \Ee_\Tt(\brho).
\ee
which is equivalent to its dual problem
\be\label{eq:LJKO_dual}
\sup_{\bphi}\inf_{\brho\geq \0}  \left\langle \brho^{n-1} - \brho, \bphi \right\rangle_\Tt - 
\frac\tau2 \sum_{\substack{\sig \in \Sig\\\sig = K|L}} a_\sig \rho_\sig (\phi_K - \phi_L)^2 + \Ee_\Tt(\brho).
\ee
Our strategy for the practical computation of the solution to~\eqref{eq:LJKO_0} is to solve the system 
corresponding to the optimality conditions of~\eqref{eq:LJKO_dual}. So far, we did not take advantage 
of the upwind choice of the mobility~\eqref{eq:upwind-1} (we only used the linearity of $(\brho, \bphi) \mapsto \left(\rho_\sig\right)_{\sig \in \Sig}$ 
in the proofs of Lemmas~\ref{lem:Legendre} and~\ref{lem:convex}, which also holds true for a centered choice of the mobilities). 
The upwinding will be key in the proof of the following theorem, which, roughly speaking, 
states that there is no need of a Lagrange multiplier for the constraint $\brho \geq \0$.

\begin{thm}\label{thm:existence}
The unique solution $(\brho^n,\bphi^n)$ to system 
\begin{equation}\label{LJKOh}
\begin{cases}
\ds m_K \phi^n_K +\frac{\tau}{2} \sum_{\sigma \in \Sigma_{K}} a_\sig \big((\phi^n_K-\phi^n_L)^+\big)^2= \frac{\p\Ee_\Tt}{\partial \rho_K}(\brho^n),  \\[15pt]
\ds (\rho^n_K-\rho^{n-1}_K) m_K + \tau\sum_{\sigma \in \Sigma_K} a_\sig \rho^n_{\sigma} (\phi^n_K-\phi^n_L) = 0,
\end{cases}
\quad \forall K \in \Tt,
\end{equation}
where $\rho_\sig^n$ denotes the upwind value, i.e.,
\[\rho^n_{\sigma} =
\begin{cases}
\rho^n_K &\text{if}\,\,  \phi^n_K > \phi^n_L, \\
\rho^n_L &\text{if}\,\,  \phi^n_K < \phi^n_L, 
\end{cases}
\quad
\forall \sigma = K|L \in \Sigma,
\]
is a saddle point of~\eqref{eq:LJKO_dual}.
\end{thm}

System~\eqref{LJKOh} is the discrete counterpart of~\eqref{eq:LJKO-oc}, whose derivation relied on the monotonicity 
of the inverse of the operator $\phi \mapsto \phi  + \frac\tau2|\grad\phi|^2$. 
Before proving Theorem~\ref{thm:existence}, let us show that the space discretization preserves 
this property at the discrete level. To this end, we introduce the functional $\bGg = \left(\Gg_K\right)_K \in C^1(\R^{\Tt};\R^{\Tt})$ defined by 
\[
\Gg_K(\bphi) := \phi_K + \frac{\tau}{2m_K} \sum_{\substack{\sig \in\Sigma_K\\\sig = K|L}} a_\sig \left( (\phi_K - \phi_L)^+\right)^2 , 
\qquad \forall K\in\Tt.
\]
\begin{lem}\label{lem:HJ}
Given $\bff \in \R^{\Tt}$, there exists a unique solution to $\bGg(\bphi) = \bff$, and it satisfies 
\be\label{eq:max_principle}
\min \bff \leq \bphi \leq \max \bff.
\ee
Moreover, let $\bphi, \wt\bphi$ be the solutions corresponding to $\bff$ and $\wt\bff$ respectively, 
then 
\be\label{eq:monotonie}
\bff\geq \wt\bff \quad \implies \quad \bphi\geq\wt\bphi.
\ee
\end{lem}
\begin{proof}
Given $\bff\geq\wt\bff$ and $\bphi, \wt\bphi$ corresponding solutions, let $K^*$ be the cell such that
\[
\phi_{K^*} - \tilde{\phi}_{K^*} = \min_{K\in \mathcal{T}} \big( \phi_K - \tilde{\phi}_K \big).
\]
Then, for all the neighboring cells $L$ of $K^*$, it holds $\phi_{K^*} - \tilde{\phi}_{K^*} \le \phi_L - \tilde{\phi}_L$ and therefore $\phi_{K^*} -\phi_L  \le  \tilde{\phi}_{K^*} - \tilde{\phi}_L$ which implies
\be\label{eq=justela}
\frac{\tau}{2m_K} \sum_{\substack{\sig \in\Sig_{K^*}\\\sig = K^*|L}} a_\sig \left( (\phi_{K^*} -\phi_L )^+\right)^2 \leq 
\frac{\tau}{2m_K} \sum_{\substack{\sig \in\Sig_{K^*}\\\sig = K^*|L}}a_\sig \left( (\tilde{\phi}_{K^*} - \tilde{\phi}_L)^+\right)^2.
\ee
Recall $\bff\geq\wt\bff$ so $\Gg_{K^*}(\bphi) \ge \Gg_{K^*}(\tilde{\bphi})$ together with \eqref{eq=justela} it yields $\phi_{K^*} \ge \tilde{\phi}_{K^*}$.
Finally as in $K^*$ the difference $\phi_K - \tilde{\phi}_K$ is minimal, we obtain $\phi_K \ge \tilde{\phi}_K$ for all $K \in \mathcal{T}$. The uniqueness of the solution $\bphi$ of $\bGg(\bphi) = \bff$ follows directly.
The maximum principle \eqref{eq:max_principle} is also a straightforward consequence of
\eqref{eq:monotonie} as one can compare $\bphi$ to $(\min \bff)\1$ and $(\max \bff)\1$ 
which are fixed points of $\bGg$. Finally, existence follows from Leray-Schauder fixed-point 
theorem~\cite{LS34} as the bounds \eqref{eq:max_principle} are uniform whatever $\tau \geq0$.
\end{proof}

With Lemma~\ref{lem:HJ} at hand, we can now prove Theorem~\ref{thm:existence}.

\begin{proof}[Proof of Theorem~\ref{thm:existence}]
Uniqueness of the solution $\brho^n$ to~\eqref{eq:LJKO_0} was already proved in Theorem~\ref{thm:LJKO}.
Owing to~\eqref{eq:NRJ_0}, $\Psi_\Tt(\brho^n; \brho^{n-1}- \brho^n )$ is finite. So Lemma~\ref{lem:Legendre} 
ensures the existence of a discrete Kantorovitch potential $\bphi^n$ satisfying (after a suitable rescaling by $\tau^{-1}$)
\be\label{eq:LJKOh-oc2}
(\rho^n_K-\rho^{n-1}_K) m_K + \tau \sum_{\sigma \in \Sigma_K} a_{\sigma} \rho^n_{\sigma} (\phi^n_K-\phi^n_L) = 0, 
\qquad \forall K \in \Tt.
\ee
The above condition is the optimality condition w.r.t. $\bphi$ in~\eqref{eq:LJKO_dual}. The optimality condition w.r.t. $\brho$ 
writes 
\be\label{eq:LJKOh-oc1}
m_K \phi^n_K +\frac{\tau}{2} \sum_{\sigma \in \Sigma_{0,K}} a_{\sigma} \big((\phi^n_K-\phi^n_L)^+\big)^2 
= \frac{\partial \Ee_\Tt}{\partial \rho_K}(\brho^n) + m_K \pi_K^n, \qquad \forall K \in \Tt
\ee
where the Lagrange multiplier $\bpi^n = \left(\pi_K^n\right)_{K} \leq \0$  for the constraint $\brho \geq \0$ satisfies $\langle \bpi^n,\brho^n\rangle_\Tt = 0$. 
To prove Theorem~\ref{thm:existence}, we thus has to prove that one can set $\bpi^n=\0$ and that the solution to system~\eqref{LJKOh} is unique. 

Let $(\brho^n,\wt\bphi^n)$ be a solution to~\eqref{eq:LJKOh-oc2}--\eqref{eq:LJKOh-oc1} with some possibly 
non-zero Lagrange multiplier $\bpi^n\leq\0$.
Denote by 
\[
\Zz^n = \{K\in\Tt\;|\;\rho_K^n=0\}, \qquad \Pp^n = \{K\in\Tt\;|\;\rho_K^n>0\} = {(\Zz^n)}^c.
\] 
Then we deduce from~\eqref{eq:LJKOh-oc2} and from the upstream choice of the mobility that 
\[
0 \leq \sum_{L\in\Nn_K} a_{KL} \rho_L^n (\wt\phi_L^n - \wt\phi_K^n)^+ = -\rho_K^{n-1} \frac{m_K}{\dt} \leq 0, 
\qquad \forall K\in\Zz^n.
\]
Hence the following alternative holds 
for all the neighbours $L\in\Nn_K$ of $K\in \Zz^n$:
\be\label{eq:L_Zz^n}
L \in \Zz^n \quad \text{or}\quad \wt\phi_K^n \geq \wt\phi_L^n.
\ee

Let $\bphi^n$ be such that $\bphi^n \geq \wt\bphi^n$ with $\phi_K^n = \wt\phi_K^n \;\text{on}\;\Pp^n$, 
then it follows from the upwinding that~\eqref{eq:LJKOh-oc2} holds for all $K\in\Tt$, while 
\eqref{eq:LJKOh-oc1} holds for $K\in\Pp^n$. 

Reproducing the proof of Lemma~\ref{lem:HJ}, there exists a unique solution 
$\left(\ov \phi_K^n\right)_{K\in\Zz^n}$ to the system 
\[
\Gg_K\left(\left(\ov \phi_L^n\right)_{L\in\Zz^n}\right) =\frac1{m_K} \frac{\p\Ee_\Tt}{\p\rho_K}(\brho^n), \qquad \forall K\in\Zz^n, 
\]
and, since $ \frac{\p\Ee_\Tt}{\p\rho_K}(\brho^n) \geq  \frac{\p\Ee_\Tt}{\p\rho_K}(\brho^n)+\pi_K^n$, the monotonicity of $\bGg$ yields
$\ov\phi_K^n \geq \wt\phi_K^n$ for all $K\in\Zz^n$.
Then define $\bphi^n\in\R^\Tt$ by 
\[
\phi_K^n = \begin{cases}
\ov \phi_K^n & \text{if}\;K\in\Zz^n, \\
\wt\phi_K^n&\text{if}\;K\in\Pp^n, 
\end{cases}
\]
 it follows from~\eqref{eq:L_Zz^n} that $(\brho^n,\bphi^n)$ fulfills~\eqref{eq:LJKOh-oc2}. 
Moreover, it follows from the upwinding in~\eqref{eq:LJKOh-oc1} and from~\eqref{eq:L_Zz^n} that  
Equation \eqref{eq:LJKOh-oc1} for $K\in\Pp^n$ does not depend on $\left(\bphi_K^n\right)_{K\in\Zz^n}$, 
so that $(\brho^n,\bphi^n)$ satisfies also~\eqref{eq:LJKOh-oc1}, but with $\bpi^n=\0$. 
Finally, owing to Lemma~\ref{lem:HJ}, $\bphi^n$ is unique, concluding the proof of Theorem~\ref{thm:existence}.
\end{proof}

\subsection{Comparison with the classical backward Euler discretization}\label{ssec:Euler}

The scheme~\eqref{eq:LJKO_0} is based on a ``first discretize then optimize'' approach. 
We have built a discrete counterpart of $\frac12 W_2^2$ and a discrete energy $\Ee_\Tt$, 
then the discrete dynamics is chosen in an optimal way by~\eqref{eq:LJKO_0}.
In opposition, the continuous equation~\eqref{eq:pde} can be thought as the 
Euler-Lagrange optimality condition for the steepest descent of the energy. 
A classical approach to approximate the optimal dynamics is to discretize directly~\eqref{eq:pde}, 
leading to what we call a ``first optimize then discretize'' approach. 
It is classical for the semi-discretization in time of~\eqref{eq:pde} to use a backward Euler scheme. 
If one combines this technic with upstream weighted Finite Volumes, we obtain the following fully discrete scheme:
\begin{equation}\label{eq:FV}
 (\check\rho^n_K-\rho^{n-1}_K)m_K + \tau \sum_{\sigma \in \Sigma_{K}}a_{\sigma}   \check \rho^n_{\sigma} (\check\phi^n_K-\check\phi^n_L) = 0, 
 \quad \text{with}\quad
\check\phi^n_K = \frac1{m_K}\frac{\partial \Ee_\Tt}{\partial \rho_K}(\check\brho^n),
\qquad \forall K \in \mathcal{T}.
\end{equation}
This scheme has no clear variational structure in the sense that, to our knowledge, $\check \brho^n$ is no longer the solution 
to an optimization problem. However, it shares some common features with our scheme~\eqref{eq:LJKO_0}: 
it is mass and positivity preserving as well as energy diminishing.
\begin{prop}\label{prop:FV}
Given $\brho^{n-1}\in\bbP_\Tt$, there exists at least one solution $(\check\brho^n,\check\bphi^n)\in\bbP_\Tt\times\R^\Tt$ to system~\eqref{eq:FV}, which satisfies 
\be\label{eq:EDI_FV}
\Ee_\Tt(\check \brho^n) + \frac1\tau \Psi_\Tt(\check\brho^n;\brho^{n-1}- \check \brho^n) + 
\tau \Psi_\Tt^*(\check\brho^n; \check \bphi^n) \leq \Ee_\Tt(\brho^{n-1}).
\ee
\end{prop}
\begin{proof}
Summing~\eqref{eq:FV} over $K\in\Tt$ provides directly the conservation of mass, i.e., $\langle \check \brho^n, \1\rangle_\Tt = \langle \brho^{n-1}, \1\rangle_\Tt$.
Assume for contradiction that $\Kk^n = \left\{K\in\Tt\; \middle| \; \check \rho_K^n < 0\right\} \neq \emptyset,$ then choose 
$K^\star \in \Kk^n$ such that $\check \phi_{K^\star}^n \geq \check \phi_{K}^n$ for all $K\in\Kk^n$. Then it follows from 
the upwind choice of the mobility in~\eqref{eq:FV} that 
\[
 \sum_{\substack{\sigma \in \Sigma_{K^\star}\\\sig=K|L}}a_{\sigma}   \check \rho^n_{\sigma} (\check\phi^n_{K^\star}-\check\phi^n_L) \leq 0,
 \]
 so that $\check \rho_{K^\star}^n \geq \rho_{K^\star}^{n-1}\geq 0$, showing a contradiction. Therefore, $\Kk^n = \emptyset$ and $\check \brho^n \geq \0$.
 These two {\em a priori} estimates (mass and positivity preservation) are uniform w.r.t. $\tau\geq 0$, thus they are sufficient to prove the existence 
 of a solution $(\check \brho^n,\check \bphi^n)$ to~\eqref{eq:FV} thanks to a topological degree argument~\cite{LS34}.
 
 Let us now turn to the derivation of the energy / energy dissipation inequality~\eqref{eq:EDI_FV}.
 Multiplying~\eqref{eq:FV} by $\check \phi_K^n$ and summing over $K\in\Tt$ provides 
 \[
 \langle \check \brho^n - \brho^{n-1}, \check\bphi^n \rangle_\Tt + 2 \tau  \Psi_\Tt^*(\check\brho^n; \check \bphi^n) = 0. 
 \]
 The definition of $\check \bphi^n$ and the convexity of $\Ee_\Tt$ yield
$
  \langle \check \brho^n - \brho^{n-1}, \check\bphi^n \rangle_\Tt  \geq \Ee_\Tt(\check \brho^n) - \Ee_\Tt(\brho^{n-1}).
$
Thus to prove~\eqref{eq:EDI_FV}, it remains to check that 
\be\label{eq:=Dissip_FV}
\frac1\tau \Psi_\Tt(\check\brho^n; \brho^{n-1}-\check \brho^n) = \tau \Psi_\Tt^*(\check\brho^n; \check \bphi^n)=  \frac1\tau \Psi_\Tt^*(\check\brho^n; \tau\check \bphi^n).
\ee
In view of~\eqref{eq:Kanto_disc}, $\tau \check \bphi^n$ is a discrete Kantorovitch potential sending $\brho^{n-1}$ on $\check \brho^{n}$ for the 
mobility corresponding to $\check \brho^n$. Therefore \eqref{eq:=Dissip_FV} holds as a consequence of~\eqref{eq:=Dissip}.
\end{proof}

Next proposition provides a finer energy / energy dissipation estimate than~\eqref{eq:NRJ_0}, 
which can be thought as discrete counterpart to the energy / energy dissipation inequality (EDI) which is a characterization 
of generalized gradient flows~\cite{AGS08,Mie11}.
\begin{prop}\label{prop:EDI}
Given $\brho^{n-1} \in \bbP_\Tt$, let $\brho^n$ be the unique solution to~\eqref{eq:LJKO_0} and let $\check \brho^n$ be a solution to~\eqref{eq:FV}, 
then 
\[
\Ee_\Tt(\brho^n) + \tau \Psi_\Tt^*\left(\brho^n; \bphi^n\right) + \tau \Psi_\Tt^*\left(\check \brho^n; \check \bphi^n\right) \leq \Ee_\Tt(\brho^{n-1}),
\]
where $\check \bphi^n$ is defined by $m_K \check \phi_K^n = \frac{\p\Ee_\Tt}{\p \rho_K}(\check \brho^n)$ for all $K\in\Tt$.
\end{prop}
\begin{proof}
Since $\check \brho^n$ belongs to $\bbP_\Tt$, it is an admissible competitor for~\eqref{eq:LJKO_0}, thus
\be\label{eq:NRJ_1}
\Ee_{\Tt}(\brho^n) + \frac1\tau \Psi_\Tt(\brho^n; \brho^{n-1}-\brho^n) \leq \Ee_{\Tt}(\check \brho^n) + \frac1\tau \Psi_\Tt(\check \brho^n;  \brho^{n-1}-\check \brho^n).
\ee
Combining this with~\eqref{eq:EDI_FV} and bearing in mind that $\frac1\tau\Psi_\Tt(\brho^n; \brho^{n-1}-\brho^{n}) = \tau \Psi_\Tt^*\left(\brho^n; \bphi^n\right)$ thanks to~\eqref{eq:=Dissip}, we obtain the desired inequality~\eqref{eq:NRJ_1}.
\end{proof}

\section{Convergence in the Fokker-Planck case}\label{sec:conv}
In this section, we investigate the limit of the scheme when the time step $\tau$ and the size of the mesh $h_\Tt$ tend to $0$ 
in the specific case of the Fokker-Planck equation~\eqref{eq:FokkerPlanck}. The size of the mesh is defined by
$h_\Tt = \max_{K\in\Tt} h_K$ with $h_K = {\rm diam}(K)$. To this end, we consider a sequence $\left(\Tt_m, \ov\Sig_m, \left(\x_K\right)_{K\in\Tt_m}\right)_{m\geq 1}$ 
of admissible discretizations of $\O$ in the sense of Section~\ref{ssec:mesh} and a sequence $\left(\tau_m\right)_{m\geq 1}$ of time steps 
such that $\lim_{m\to\infty}\tau_m = \lim_{m\to\infty}h_{\Tt_m} = 0.$
We also make the further assumptions on the mesh sequence: there exists $\zeta>0$ such that, for all $m\geq 1$, 
\begin{subequations}\label{eq:reg_mesh}
\be\label{eq:reg_1}
h_K \leq \zeta d_\sig \leq \zeta^2 h_K, \qquad \forall \sig \in \Sig_K, \; \forall K\in\Tt_m, 
\ee
\be\label{eq:reg_2}
{\rm dist}(\x_K,\ov K) \leq \zeta h_K, \qquad \forall K \in \Tt_m, 
\ee
and
\be\label{eq:reg_3}
\sum_{\sig \in \sig_K} m_{\Delta_\sig} \leq \zeta m_K, \qquad \forall K \in \Tt_m.
\ee
\end{subequations}

Let $T>0$ be an arbitrary finite time horizon, then we assume for the sake of simplicity that 
$\tau_m = T/N_m$ for some integer $N_m$ tending to $+\infty$ with $m$.
For the ease of reading, we remove the subscript $m\geq 1$ when it appears to be unnecessary for understanding. 

Given $V \in C^2(\ov\O)$, we define the discrete counterpart of the energy~\eqref{eq:NRJ_FP} by 
\[
\Ee_{\Tt}(\brho)= \sum_{K\in\Tt} m_K \left[\rho_K \log\frac{\rho_K}{e^{-V_K}} - \rho_K + e^{-V_K}\right], \qquad \forall \brho \in \R_+^{\Tt}, 
\]
where $V_K = V(\x_K)$ for all $K\in\Tt$. In view of the above formula, there holds
\be\label{eq:FP_dE}
\frac{\p\Ee_{\Tt}}{\p\rho_K}(\brho) = m_K (\log(\rho_K) + V_K) \qquad \forall K\in\Tt.
\ee
Given an initial condition $\rhoe^0 \in \bbP(\O)$ with positive mass, i.e. $\int_\O \rhoe^0 \d\x >0$, and such that $\Ee(\rhoe^0) < \infty$, it 
is discretized into $\brho^0 = \left(\rho_K^0\right)_{K\in\Tt}$ defined by 
\be\label{eq:rhoK0}
\rho_K^0 = \frac1{m_K} \int_K \rhoe^0 \d\x \geq 0, \qquad \forall K\in \Tt.
\ee
Note that the energy $\Ee_\Tt$ is not in $C^1(\R^\Tt_+)$ since its gradient blows up on $\p\R_+^\Tt$. However, the functional $\Ee_\Tt$ 
is continuous and strictly convex on $\R^\Tt_+$, hence the scheme~\eqref{eq:LJKO_0} still admits a unique solution $\brho^n$ 
for all $n \geq 1$ thanks to Theorem~\ref{thm:LJKO}, since its proof does not use the differentiability of the energy. 
Thanks to the conservativity of the scheme and definition~\eqref{eq:rhoK0} of $\brho^0$, one has 
\[
\langle \brho^n, \1\rangle_{\Tt} = \langle \brho^0, \1\rangle_{\Tt} = \int_\O \rhoe^0 \d\x>0, \qquad \forall n \geq 1.
\]
Let us show that $\brho^n >\0$ for all $n\geq 1$.
To this end, we proceed as in~\cite[Lemma 8.6]{Santambrogio_OTAM}.
\begin{lem}\label{lem:rho_pos}
Assume that $\rhoe^0$ has positive mass, then the iterated solutions $\left(\brho^n\right)_{n\geq1}$ to scheme~\eqref{eq:LJKO_0} 
satisfy $\brho^n>\0$ for all $n\geq 1$. Moreover, there exists a unique sequence $\left(\bphi^n\right)_{\geq1}$ of discrete Kantorovitch potentials such that 
the following optimality conditions are satisfied for all $K\in\Tt$ and all $n \geq 1$:
\begin{align}
& \phi_K^n +\frac\tau{2m_K} \sum_{\sig =K|L\in \Sig_K} a_\sig \left( (\phi_K^n - \phi_L^n)^+ \right)^2 = \log(\rho_K^n) + V_K, \label{eq:FP_HJ}
\\
& (\rho_K^n-\rho_K^{n-1}) m_K + \tau \sum_{\sig=K|L \in \Sig} a_\sig \rho_\sig^n (\phi_K^n - \phi_L^n) = 0.\label{eq:FP_cons}
\end{align} 
\end{lem}
\begin{proof}
Define $\ov \rho =\frac1{|\O|}\int_\O \rhoe^0\d\x$ and  $\ov\brho = \ov \rho \1 \in \bbP_\Tt$, and by $\brho^n_\eps = \left(\rho_{K,\eps}^n\right)_{K\in\Tt}
= \eps \ov \brho + (1-\eps) \brho^n \in \bbP_\Tt$ for some arbitrary $\eps \in (0,1)$.
Since $\brho^n$ is optimal in~\eqref{eq:LJKO_0}, there holds 
\begin{multline}\label{eq:toto}
\sum_{K\in\Tt} m_K\left[ \rho_K^n \log\rho_K^n - \rho_{K,\eps}^n \log\rho_{K,\eps}^n\right]\leq \sum_{K\in\Tt}m_K \left(\rho_{K,\eps}^n - \rho_K^n\right) V_K \\
+ \Psi_\Tt(\brho_\eps^n;  \brho^{n-1}-\brho^n_{\eps}) - \Psi_\Tt(\brho^n;\brho^{n-1}- \brho^n).
\end{multline}
The convexity of $\brho\mapsto \Psi_\Tt(\brho, \brho^{n-1}-\brho)$ implies that 
\[
\Psi_\Tt(\brho_\eps^n;  \brho^{n-1}-\brho^n_{\eps}) \leq \eps \Psi_\Tt(\ov \brho;\brho^{n-1}- \ov \brho ) + (1-\eps) \Psi_\Tt(\brho^n;  \brho^{n-1}-\brho^n), 
\]
while the boundedness of $V$ provides 
\[
\sum_{K\in\Tt}m_K \left(\rho_{K,\eps}^n - \rho_K^n\right) V_K \leq \eps \|V\|_{L^\infty(\O)} \|\rhoe^0\|_{L^1(\O)}.
\]
Therefore, the right-hand side in~\eqref{eq:toto} can be overestimated by 
\[
\sum_{K\in\Tt} m_K\left[ \rho_K^n \log\rho_K^n - \rho_{K,\eps}^n \log\rho_{K,\eps}^n\right]\\ 
\leq C \eps
\]
for some $C$ depending on $\brho^n, \brho^{n-1}$ and $V$ but not on $\eps$. Setting $\Zz^n = \{K\in \Tt\;| \; \rho_K^n = 0\}$ 
and $\Pp^n =  \{K\in \Tt\;|\; \rho_K^n > 0\} = {(\Zz^n)}^c$, we have 
\[
\sum_{K\in\Zz^n} m_K\left[ \rho_K^n \log\rho_K^n - \rho_{K,\eps}^n \log\rho_{K,\eps}^n\right] = 
\eps \sum_{K\in\Zz^n} m_K \ov \rho  \log \eps \ov \rho, 
\]
and, thanks to the convexity of $\rho\mapsto \rho \log \rho$ and to the monotonicity of $\rho\mapsto \log\rho$, 
\begin{align*}
\sum_{K\in\Pp^n} m_K\left[ \rho_K^n \log\rho_K^n - \rho_{K,\eps}^n \log\rho_{K,\eps}^n\right] \geq & 
\;\eps \sum_{K\in\Pp^n} m_K (\rho_K^n - \ov \rho)(1+ \log(\rho_{K,\eps}^n)) \\
\geq & \;\eps\sum_{K\in\Pp^n} m_K (\rho_K^n - \ov \rho)(1+ \log(\ov \rho)) \geq - C \eps. 
\end{align*}
Then dividing by $\eps$ and letting $\eps$ tend to $0$, we obtain that 
\[
\underset{\eps\to0}{\text{limsup}} \sum_{K\in\Zz^n} m_K \ov \rho  \log \eps \ov \rho \leq C, 
\]
which is only possible if $\Zz^n = \emptyset$, i.e., $\brho^n > \0$. This implies that $\Ee_\Tt$ is differentiable at $\brho^n$, 
hence the optimality conditions~\eqref{LJKOh} hold, which rewrites as~\eqref{eq:FP_HJ}--\eqref{eq:FP_cons} thanks to~\eqref{eq:FP_dE}. 
By the way, the uniqueness of the discrete Kantorovitch potential $\bphi^n$ 
for all $n \geq 1$ is provided by Theorem~\ref{thm:existence}.
\end{proof}

Lemma~\ref{lem:rho_pos} allows to define 
two functions $\rho_{\Tt,\tau}$ and $\phi_{\Tt,\tau}$ by setting
\[
\rho_{\Tt,\tau}(\x,t) = \rho_K^n, \quad \phi_{\Tt,\tau}(\x,t) = \phi_K^n \quad \text{if}\; (\x,t) \in K\times(t^{n-1},t^n].
\]
It follows from the conservativity of the scheme and definition~\eqref{eq:rhoK0} of $\brho^0$ that 
\[
\int_{\O} \rho_{\Tt,\tau}(\x,t^n) \d\x = \langle \brho^n, \1\rangle_{\Tt} = \langle \brho^0, \1\rangle_{\Tt} = \int_\O \rhoe^0 \d\x>0,
\]
so that $\rho_{\Tt,\tau}(\cdot,t)$ belongs to $\bbP(\O)$ for all $t\in (0,T)$. 

The goal of this section is to prove the following theorem.
\begin{thm}\label{thm:conv}
Assume that $\rhoe^0 \geq \rho_\star$ for some $\rho_\star \in (0,+\infty)$ and that $\Ee(\rhoe^0) < +\infty$, 
and let $\left(\Tt_m, \ov\Sig_m, \left(\x_K\right)_{K\in\Tt_m} \right)_{m\geq 1}$ be a sequence of admissible discretizations of $\O$ 
 such that $h_{\Tt_m}$ and $\tau_m$ tend to $0$ while conditions~\eqref{eq:reg_mesh} hold. 
Then up to a subsequence, $\left(\rho_{\Tt_m,\tau_m}\right)_{m\geq 1}$ tends in $L^1(Q_T)$ towards 
a weak solution $\rhoe \in L^\infty((0,T);L^1(\O)) \cap L^2((0,T);W^{1,1}(\O))$ 
of~\eqref{eq:FokkerPlanck} corresponding to the initial data $\rhoe^0$.
\end{thm}
The proof is based on compactness arguments. At first in Section~\ref{ssec:apriori}, 
we derive some a priori estimates on the discrete solution. These estimates will be used to obtain 
some compactness on $\rho_{\Tt_m,\tau_m}$ and $\phi_{\Tt_m,\tau_m}$ in Section~\ref{ssec:compact}. 
Finally, we identify the limit value as a weak solution in Section~\ref{ssec:identify}.

\subsection{Some a priori estimates}\label{ssec:apriori}

First, let us show that if the continuous initial energy $\Ee(\rhoe^0)$ is bounded, then so does its discrete counterpart $\Ee_\Tt(\brho^0)$. 
\begin{lem}\label{lem:Ee_Tt0}
Given $\rhoe^0 \in \bbP(\O)$ such that $\Ee(\rhoe^0)<+\infty$, and let $\brho^0$ be defined by~\eqref{eq:rhoK0}, then 
there exists $\ctel{cte:L2H1_phi_1}$ depending only on $\O$, $V$ and $\rhoe^0$ (but not on $\Tt$) such that 
$\Ee_\Tt(\brho^n) \leq \cter{cte:L2H1_phi_1}$ for all $n \geq 0$.
\end{lem}
\begin{proof}
It follows from~\eqref{eq:NRJ_0} that $\Ee_\Tt(\brho^n) \leq \Ee_\Tt(\brho^0)$ for all $n\geq 1$. 
Rewriting $\Ee_\Tt(\brho^0)$ as 
\be\label{eq:T123}
\Ee_\Tt(\brho^0) = T_1 + T_2 + T_3  
\ee
with 
\[
T_1  =  \sum_{K\in\Tt} m_K [\rho_K^0 \log \rho_K^0 - \rho_K^0] , \quad
T_2  =  \sum_{K\in\Tt} m_K \rho_K^0 V_K, 
\quad \text{and}\quad T_3= \sum_{K\in\Tt} m_K e^{-V_K}, 
\]
we deduce from the definition~\eqref{eq:rhoK0} of $\brho^0$ and Jensen's inequality that 
\be\label{eq:T1}
T_1 \leq \int_\O [\rhoe^0 \log\rhoe^0 - \rhoe^0] \d\x. 
\ee
Since $V$ is continuous, there exists $\wt \x_K \in K$ such that $\int_K e^{-V}\d\x = m_K e^{-V(\wt \x_K)}$. 
Therefore, 
\be\label{eq:T3}
T_3 = \int_\O e^{-V} \d\x + \sum_{K\in\Tt} m_K [e^{-V(\x_K)} - e^{-V(\wt \x_K)}] \leq  
\int_\O e^{-V} \d\x + e^{\|V^-\|_\infty} \|\grad V\|_\infty {\rm diam}(\O). 
\ee
Similarly, it follows from the mean value theorem that there exists $\check \x_K \in K$ such that 
$
m_K V(\check \x_K) \rho_K^0 = \int_K \rhoe^0 V \d\x. 
$
Hence, 
\be\label{eq:T2}
T_2 = \int_\O \rhoe^0 V \d\x + \sum_{K\in\Tt}m_K \rho_K^0 [V(\x_K)-V(\check \x_K)] \leq \int_\O \rhoe^0 V \d\x
+ \|\grad V\|_\infty {\rm diam}(\O) \int_\O \rhoe^0 \d\x.
\ee
Combining~\eqref{eq:T1}--\eqref{eq:T2} in~\eqref{eq:T123} shows that $\Ee_\Tt(\brho^0) \leq \Ee(\rhoe^0) + C$ 
for some $C$ depending only on $V$, $\O$ and $\rhoe^0$. 
\end{proof}

Our next lemma shows that if $\rhoe^0$ is bounded away from $0$, then so does $\rho_{\Tt,\tau}$.
\begin{lem}\label{lem:inf_phi}
Using the convention $\log(0) = - \infty$, one has 
\[
\min_{K\in\Tt} \left[\log(\rho_K^n) + V_K\right] \geq \min_{K\in\Tt} \left[\log(\rho_K^{n-1}) + V_K\right], \qquad \forall n \geq 1. 
\]
In particular, if $\rhoe^0\geq \rho_\star$ for some $\rho_\star \in (0,+\infty)$, then there exists $\alpha>0$ depending 
only on $V$ and $\rho_\star$ (but not on $\Tt,\tau$ and $n$) such that $\brho^n \geq \alpha \1$ for all $n \geq 1$.
\end{lem}
\begin{proof}
It follows directly from~\eqref{eq:FP_HJ} that 
$
\log(\rho_K^n) + V_K \geq \phi_K^n
$
for all $K\in\Tt$.
Let $K_\star \in \Tt$ be such that $\phi_{K_\star}^n \leq \phi_K^n$ for all $K\in\Tt$, then the conservation equation~\eqref{eq:FP_cons} 
ensures that $\rho_{K_\star}^n \geq \rho_{K_\star}^{n-1}$. On the other hand, since 
\[\sum_{\sig =K_\star|L\in \Sig_{K_\star}} a_\sig \left( (\phi_{K_\star}^n - \phi_L^n)^+ \right)^2 =0,\]
the discrete HJ equation \eqref{eq:FP_HJ} provides that 
\[
\phi_{K^\star}^n =  \log(\rho_{K_\star}^n) + V_{K_\star} = \min_{K\in\Tt} \left[\log(\rho_K^n) + V_K\right]\geq \log(\rho_{K_\star}^{n-1}) + V_{K_\star} \geq \min_{K\in\Tt} \left[\log(\rho_K^{n-1}) + V_K\right].
\]
Assume now that $\rhoe^0 \geq \rho_\star$, then for all $K\in\Tt$ and all $n\geq 0$, 
\[
\log(\rho_K^n) \geq \min_{L\in\Tt}[ \log(\rho_L^0) + V_L]- V_K \geq \min_{L\in\Tt}\log(\rho_L^0) - 2 \|V\|_\infty \geq \log(\rho_\star) - \|V^+\|_\infty - \|V^-\|_\infty.
\]
Therefore, we obtain the desired inequality with $\alpha = \rho_\star e^{- \|V^+\|_\infty - \|V^-\|_\infty}$.
\end{proof}

Our third lemma deals with some estimates on the discrete gradient of the discrete Kantorovitch potentials $\left(\bphi^n\right)_n$.
\begin{lem}\label{lem:L2H1_phi}
Let $(\brho^n, \bphi^n)$ be the iterated solution to~\eqref{LJKOh}, then
\be\label{eq:L2H1_phi_1}
 \sum_{n = 1}^N \tau \sum_{\sig=K|L\in\Sig} a_\sig \rho_\sig^n (\phi_K^n - \phi_L^n)^2 \leq  \cter{cte:L2H1_phi_1}. 
\ee
Moreover, if $\rhoe^0 \geq \rho_\star \in (0,+\infty)$, then there exists $\ctel{cte:L2H1_phi_2}$ depending on $\O$, $V$ and $\rhoe^0)$ 
such that 
\be\label{eq:L2H1_phi_2}
\sum_{n = 1}^N \tau \sum_{\sig=K|L\in\Sig} a_\sig  (\phi_K^n - \phi_L^n)^2 \leq \cter{cte:L2H1_phi_2}. 
\ee
\end{lem}
\begin{proof}
Since $\Ee_\Tt(\brho)\geq 0$ for all $\brho \in \bbP_\Tt$, summing~\eqref{eq:NRJ_0} over $n\in\{1,\dots, N\}$ yields 
\[
\sum_{n = 1}^N \frac1\tau \Psi_\Tt(\brho^n;  \brho^{n-1}-\brho^n) \leq \Ee_\Tt(\brho^0).
\]
Thanks to~\eqref{eq:=Dissip}, the left-hand side rewrites 
\[
\sum_{n = 1}^N \frac1\tau \Psi_\Tt(\brho^n; \brho^{n-1} -\brho^n) = \sum_{n = 1}^N \tau \sum_{\sig=K|L\in\Sig} a_\sig \rho_\sig^n (\phi_K^n - \phi_L^n)^2, 
\]
so that it only remains to use Lemma~\ref{lem:Ee_Tt0} to recover~\eqref{eq:L2H1_phi_1}.

Finally, if $\rhoe^0$ is bounded from below by some $\rho_\star>0$, then Lemma~\ref{lem:inf_phi} shows that $\rho_K^n \geq \alpha$ 
for some $\alpha$ depending only on $\rho_\star$ and $V$. Therefore, since $\rho_\sig^n$ is either equal to $\rho_K^n$ or to $\rho_L^n$ 
for $\sig = K|L \in \Sig$, then~\eqref{eq:L2H1_phi_2} holds with $\cter{cte:L2H1_phi_2} = \frac{\cter{cte:L2H1_phi_1}}\alpha$.
\end{proof}

The discrete solution $\rho_{\Tt,\tau}$ is piecewise constant on the cells. To study the convergence of the scheme, we also need a
second reconstruction $\rho_{\Sig,\tau}$  of the density corresponding to the edge mobilities. It is defined by 
\be\label{eq:rho_Sig_def}
\rho_{\Sig,\tau}(\x,t) = \begin{cases}
\rho_\sig^n & \text{if}\; (\x,t) \in \Delta_\sig \times (t^{n-1},t^n], \quad \sig \in \Sig, \\
\rho_K^n & \text{if}\; (\x,t) \in K \setminus \left( \bigcup_{\sig\in\Sig_K} \Delta_\sig \right)  \times (t^{n-1},t^n], \quad  K \in \Tt. 
\end{cases}
\ee
\begin{lem}\label{lem:rho_Sig}
There exists $\ctel{cte:rho_Sig}$ depending only on $\zeta$ and $\rhoe^0$ such that 
\be\label{eq:rho_Sig}
\int_\O \rho_{\Sig,\tau}(\x,t) \d\x \leq   \cter{cte:rho_Sig}, \qquad \forall t > 0. 
\ee
Moreover, there exists $\ctel{cte:H_Sig}$ depending only on $\zeta, V$ and $\rhoe^0$ such that 
\be\label{eq:H_Sig}
\int_\O \rho_{\Sig,\tau}(\x,t) \log \rho_{\Sig,\tau}(\x,t) \d\x \leq   \cter{cte:H_Sig}, \qquad \forall t>0.
\ee
\end{lem}
\begin{proof}
Since $t\mapsto  \rho_{\Sig,\tau}(\cdot,t)$ is piecewise constant, it  suffices to check that the above properties at each $t^n$, $1\leq n \leq N$. 
In view of the definition of $\rho_{\Sig,\tau}$, one has 
\[
\int_\O \rho_{\Sig,\tau}(\x,t^n) \d\x \leq  \sum_{K\in\Tt} \sum_{\sig\in\Sig_K\cap\Sig_{\rm ext}} \rho_K^n m_K +
\sum_{\sig \in \Sig} \rho_\sig^n m_{\Delta_\sig}. 
\]
The first term can easily be overestimated by $\int_\O \rho_{\Tt,\tau}(\x,t^n) \d\x = \int_\O \rhoe^0 \d\x$. 
Since $\rho_\sig^n \leq \rho_K^n + \rho_L^n$, the second term in the above expression can be overestimated by 
\[
\sum_{\sig \in \Sig} \rho_\sig^n m_{\Delta_\sig} \leq \sum_{K\in\Tt}\rho_K^n \left( \sum_{\sig \in \Sig_K} m_{\Delta_\sig} \right).
\]
Using the regularity property of the mesh~\eqref{eq:reg_3},  
we obtain that 
\[
\sum_{\sig \in \Sig} \rho_\sig^n m_{\Delta_\sig} \leq \zeta \int_\O \rhoe^0 \d\x, 
\]
so that \eqref{eq:rho_Sig} holds with $\cter{cte:rho_Sig}= (1+\zeta)\int_\O \rhoe^0\d\x$.

Reproducing the above calculations, one gets that 
\begin{align*}
\int_\O \rho_{\Sig,\tau}(\x,t) \log \rho_{\Sig,\tau}(\x,t) \d\x \leq&\;(1+\zeta) \int_\O \rho_{\Tt,\tau}(\x,t) \log \rho_{\Tt,\tau}(\x,t) \d\x\\
= &\;(1+\zeta)\left(\Ee_\Tt(\brho^n) + \sum_{K\in\Tt} m_K [ \rho_K^n(1-V_K) - e^{-V_K}] \right).  
\end{align*}
Since $\Ee_\Tt(\brho^n)\leq \Ee_\Tt(\brho^0) \leq \cter{cte:L2H1_phi_1}$ and since $V$ is uniformly bounded, we obtain
that \eqref{eq:H_Sig} holds with $\cter{cte:H_Sig} = (1+\zeta) \left(\cter{cte:L2H1_phi_1} + \|(1-V)^+\|_\infty\right)$.
\end{proof}

The last lemma of this section can be thought as a discrete $\left(L^\infty((0,T);W^{1,\infty}(\O))\right)'$ estimate on $\p_t \rho_{\Tt,\tau}$.
This estimate will be used to apply a discrete nonlinear Aubin-Simon lemma~\cite{ACM17} in the next section.

\begin{lem}\label{lem:drho_dt}
Let $\varphi \in C^\infty_c(Q_T)$, then define $\varphi_K^n = \frac1{m_K}\int_K \varphi(\x,t^n)\d\x$ for all $K\in\Tt$.
There exists $\ctel{cte:drho_dt}$ depending only on $\zeta, T,\rhoe^0, d$,  such that 
\[
\sum_{n=1}^N \sum_{K\in\Tt} m_K (\rho_K^n - \rho_K^{n-1}) \varphi_K \leq \cter{cte:drho_dt} \|\grad \varphi \|_{L^\infty(Q_T)}.
\]
\end{lem}
\begin{proof}
Multiplying~\eqref{eq:FP_cons} by $\varphi_K^n$ and summing over $K\in\Tt$ and $n\in\{1,\dots, N\}$ yields
\[
A := \sum_{n=1}^N \sum_{K\in\Tt} m_K (\rho_K^n - \rho_K^{n-1}) \varphi_K = - \sum_{n=1}^N \tau 
\sum_{\sig = K|L \in \Sig} a_\sig \rho_\sig^n (\phi_K^n - \phi_L^n) (\varphi_K^n - \varphi_L^n).
\]
Applying Cauchy-Schwarz inequality on the right-hand side then provides
\be\label{eq:A^2}
A^2 \leq \left(\sum_{n=1}^N \tau \sum_{\sig = K|L \in \Sig} a_\sig \rho_\sig^n (\phi_K^n - \phi_L^n)^2 \right)
\left(\sum_{n=1}^N \tau \sum_{\sig = K|L \in \Sig} a_\sig \rho_\sig^n (\varphi_K^n - \varphi_L^n)^2 \right).
\ee
The first term in the right-hand side is bounded thanks to Lemma~\ref{lem:L2H1_phi}. On the other hand, 
the regularity of $\varphi$ ensures that there exists $\wt \x_K \in K$ such that $\varphi(\x_K,t^n) = \varphi_K^n$ for all $K\in\Tt$.
Thanks to the regularity assumptions~\eqref{eq:reg_1}--\eqref{eq:reg_2} on the mesh, there holds
\[
|\varphi_K^n - \varphi_L^n| \leq \|\grad \varphi\|_{\infty}|\wt\x_K - \wt\x_L| \leq  (1+2\zeta(1+\zeta)) \|\grad \varphi\|_{\infty} d_\sig, \qquad \sig = K|L.
\]
Hence, the second term of the right-hand side in~\eqref{eq:A^2} can be overestimated by 
\begin{align*}
\sum_{n=1}^N \tau \sum_{\sig = K|L \in \Sig} a_\sig \rho_\sig^n (\varphi_K^n - \varphi_L^n)^2 
\leq&\; (1+2\zeta(1+\zeta))^2 \|\grad \varphi\|^2_\infty \sum_{n=1}^N \tau \sum_{\sig = K|L \in \Sig} m_\sig d_\sig \rho_\sig^n \\
\leq&\; (1+2\zeta(1+\zeta))^2 d \|\grad \varphi\|^2_\infty \iint_{Q_T}\rho_{\Sig,\tau}\d\x\d t \\
\leq&\; (1+2\zeta(1+\zeta))^2  \cter{cte:rho_Sig} T d \|\grad \varphi\|_\infty^2, 
\end{align*}
the last inequality being a consequence of Lemma~\ref{lem:rho_Sig}.
Combining all this material in~\eqref{eq:A^2} shows the desired estimate with 
$\cter{cte:drho_dt}= (1+2\zeta(1+\zeta)) \sqrt{\cter{cte:L2H1_phi_1} \cter{cte:rho_Sig} T d}$.
\end{proof}
		
\subsection{Compactness of the approximate solution}\label{ssec:compact}

The goal of this section is to show enough compactness in order to be able to pass to the limit $m\to\infty$.
For the sake of readability, we remove the subscript $m$ unless necessary. 

Owing to Lemma~\ref{lem:Ee_Tt0}, one has $\Ee_\Tt(\brho^n) \leq \cter{cte:L2H1_phi_1}$ for all $n\in\{1,\dots,N\}$. 
Proceeding as in the proof of Lemma~\ref{lem:rho_Sig}, this allows to show that 
\be\label{eq:rhologrho_Tt}
\int_\O \rho_{\Tt,\tau}(\x,t) \log \rho_{\Tt,\tau}(\x,t) \d\x \leq \cter{cte:rhologrho}, \qquad \forall t \in (0,T]
\ee
for some $\ctel{cte:rhologrho}$ depending only on $\rhoe^0$, $\zeta$ and $V$.
Combining de La Vallée Poussin's theorem with Dunford-Pettis' one \cite[Ch. XI, Theorem 3.6]{Visintin96}, 
there exists $\rhoe \in L^\infty((0,T);L^1(\O))$ such that, up to a subsequence, 
\be\label{eq:L1_weak}
\text{$\rho_{\Tt_m,\tau_m}$ tends to $\rhoe$ 
weakly in $L^1(Q_T)$ as $m$ tends to $+\infty$.}
\ee
Since $\rho\mapsto \rho\log\rho$ is convex, $f\mapsto\iint_{Q_T}f\log f\d\x\d t$ is l.s.c. for the weak convergence 
in $L^1(Q_T)$ (see for instance~\cite[Corollary 3.9]{Brezis11}), 
so that~\eqref{eq:rhologrho_Tt} yields 
\be\label{eq:rhologrho}
\iint_{Q_T} \rhoe\log\rhoe \d\x\d t \leq  \cter{cte:rhologrho} T.
\ee
Moreover, since $\rho_{\Tt,\tau} \geq \alpha$ thanks to Lemma~\ref{lem:inf_phi}, then $\rhoe \geq \alpha$ too.

Our goal is to show that $\rhoe$ is the unique weak solution 
to the Fokker-Planck equation~\eqref{eq:FokkerPlanck} corresponding to the initial data $\rhoe^0$. 
Even though the continuous problem is linear, \eqref{eq:L1_weak} is not enough to pass to the limit in our nonlinear scheme.
Refined compactness have to be derived in this section so that one can identify $\rhoe$ as the solution to 
\eqref{eq:FokkerPlanck} in the next section. To show enhanced compactness (and most of all the 
consistency of the scheme in the next section), we have to assume that the initial data is bounded away from $0$. 

\begin{prop}\label{prop:compact}
Assume that $\rhoe^0 \geq \rho_\star \in (0,+\infty)$, then, up to a subsequence, 
\begin{eqnarray}
\label{eq:conv_rho_Tt}
&\rho_{\Tt_m,\tau_m} \underset{m\to\infty}\longrightarrow \rhoe \quad&\text{strongly in $L^1(Q_T)$},\\
\label{eq:conv_log_Tt}
&\log\rho_{\Tt_m,\tau_m} \underset{m\to\infty}\longrightarrow \log\rhoe \quad&\text{strongly in $L^1(Q_T)$},\\
\label{eq:conv_phi_Tt}
&\phi_{\Tt_m,\tau_m} \underset{m\to\infty}\longrightarrow \log\rhoe + V \quad&\text{strongly in $L^1(Q_T)$}.
\end{eqnarray}
\end{prop}
\begin{proof}
Our proof of~\eqref{eq:conv_rho_Tt}--\eqref{eq:conv_log_Tt} relies on ideas introduced in~\cite{Moussa16} that 
were adapted to the discrete setting in~\cite{ACM17}.
Define the two convex and increasing conjugated functions defined on $\R_+$:
\[\Upsilon: x \mapsto e^x - x - 1\quad \text{and} \quad\Upsilon^*:y\mapsto (1+y)\log(1+y) - y,\]
then the following  inequality holds for any measurable functions $f,g:Q_T \to \R$:
\be\label{eq:Young}
\iint_{Q_T}|fg|\d\x\d t\leq \iint_{Q_T}\Upsilon(|f|)\d\x\d t +  \iint_{Q_T}\Upsilon^*(|g|)\d\x\d t.
\ee
Now, notice that since $\rho_{\Tt,\tau}$ is bounded from below thanks to Lemma~\ref{lem:inf_phi} and bounded in $L^1(Q_T)$, then 
$\log\rho_{\Tt,\tau}$ is bounded in $L^p(Q_T)$ for all $p\in[1,\infty)$ and $\Upsilon(|\log(\rho_{\Tt,\tau})|)$ is bounded in $L^1(Q_T)$. 
As a consequence, there exists $\ell \in L^\infty((0,T);L^p(\O))$ 
such that 
\be\label{eq:tralala}
\log\rho_{\Tt_m,\tau_m} \underset{m\to\infty}\longrightarrow \ell \quad\text{weakly in $L^1(Q_T)$.}
\ee
Since $f \mapsto \iint_{Q_T}\Upsilon(|f|)$ is convex thus l.s.c. for the weak convergence, 
we infer that $\Upsilon(|\ell|)$ belongs to $L^1(Q_T)$. Moreover, in view of~\eqref{eq:rhologrho}, $\Upsilon^*(\rhoe)$ belongs also to $L^1(Q_T)$.
Therefore, thanks to~\eqref{eq:Young}, the function $\rhoe\ell$ is in $L^1(Q_T)$. 

Define the quantities
	\[
	r_K^n = \frac{\tau}{2 m_K} a_\sig \sum_{\sigma \in \Sigma_{K}} \big( (\phi_K^n-\phi_L^n)^+ \big)^2 \geq 0, \, \forall K \in \Tt, \; \forall n\in\{1,\dots, N\}, 
	\]
and by $r_{\Tt,\tau} \in L^1(Q_T)$ the function defined 
\[
r_{\Tt,\tau}(\x,t) = 
r_K^n \quad \text{if}\; (\x,t) \in K\times(t^{n-1},t^n], 
\]
Thanks to Lemma~\ref{lem:L2H1_phi}, ${\|r_{\Tt,\tau}\|}_{L^1(Q_T)}\le \frac12 \cter{cte:L2H1_phi_2}\tau$.
As a consequence, $r_{\Tt_m,\tau_m}$ tends to $0$ in $L^1(Q_T)$ as $m$ tends to $+\infty$. 

Let $\bxi\in\R^d$ be arbitrary, we denote by $\O_\bxi = \{\x\in\O\; |\; \x+\bxi\in\O\}$. Then using~\eqref{eq:FP_HJ} and the triangle inequality, 
we obtain that for all $m\geq 1$, there holds
\[
\int_{0}^{T} \int_{\Omega_{\bxi}} \left| \log \rho_{\Tt_m, \tau_m}(\x+\bxi, t) - \log\rho_{\Tt_m,\tau_m}(\x,t)\right|\d\x\d t \leq A_{1,m}(\bxi) + A_{2,m}(\bxi) + A_{3,m}(\bxi), 
\]
where, denoting by  
$V_{\Tt}(\x) = V_K$ if $\x \in K$, 
we have set
\begin{align*}
A_{1,m}(\bxi) =&\; \int_{0}^{T} \int_{\Omega_{\bxi}} | r_{\Tt_m,\tau_m}(\x+\bxi,t)-r_{\Tt_m,\tau_m}(\x,t)| \d\x\d t,\\
A_{2,m}(\bxi) =&\; \int_{0}^{T} \int_{\Omega_{\bxi}} |\phi_{\Tt_m,\tau_m}(\x+\bxi,t)-\phi_{\Tt_m,\tau_m}(\x,t)| \d \x \d t, \\
A_{3,m}(\bxi) =&\; T\int_{\Omega_{\bxi}} |V_{\Tt_m}(\x+\bxi)-V_{\Tt_m}(\x)| \d\x.
\end{align*}
Since $\left(r_{\Tt_m,\tau_m}\right)_{m\geq1}$ and $\left(V_{\Tt_m}\right)_{m\geq1}$ are compact in $L^1(Q_T)$ and $L^1(\O)$ respectively, 
it follows from the Riesz-Frechet-Kolmogorov theorem (see for instance \cite[Exercise 4.34]{Brezis11}) that there exists 
 $\omega\in C(\R_+;\R_+)$ with $\omega(0) = 0$ such that   
\be\label{eq:A1+A3}
A_{1,m}(\bxi) + A_{3,m}(\bxi) \le \omega(|\bxi|), \qquad \forall \bxi\in\R^d,\;\forall m \geq 0.
\ee
On the other hand, the function $\phi_{\Tt,\tau}$ belongs to $L^1((0,T);BV(\O))$ and the integral in time of its total variation in space 
can be estimated as follows:
\begin{align*}
\iint_{Q_T}| \grad \phi_{\Tt_m,\tau_m}| = &\;\sum_{n=1}^N \tau \sum_{\sig =K|L\in \Sig} m_\sig |\phi_K^n - \phi_L^n| \\
\leq &\;\left(d|\O|T  \sum_{n=1}^N \tau \sum_{\sig =K|L\in \Sig} m_\sig (\phi_K^n - \phi_L^n)^2 \right)^{1/2} \leq \cter{cte:BV}.
\end{align*}
with  $\ctel{cte:BV}=\sqrt{d|\O|T \cter{cte:L2H1_phi_2}}$.
This implies in particular that $A_{2,m}(\bxi) \le \cter{cte:BV}|\bxi|$ for all $m\geq 1$.
Combining this estimate with~\eqref{eq:A1+A3} in~\eqref{eq:FP_HJ} yields
\be\label{eq:compact_x}
\sup_{m\geq 1} \int_{0}^{T} \int_{\Omega_{\xi}} |\log \rho_{\Tt_m,\tau_m}(\x+\bxi,t)-\log \rho_{\Tt_m,\tau_m}(\x,t)| \d\x\d t \underset{|\bxi|\to 0}\longrightarrow 0.
\ee
The combination of~\eqref{eq:compact_x} with Lemma~\ref{lem:drho_dt} is exactly what one needs to reproduce the proof of~\cite[Proposition 3.8]{ACM17}, 
which shows that the product of the weakly convergent sequences ${(\rho_{\Tt_m,\tau_m})}_m$ and ${(\log\rho_{\Tt_m,\tau_m})}_{m}$ 
converges towards the product of their weak limits:
\be\label{eq:Moussa}
\iint_{Q_T}\rho_{\Tt_m,\tau_m} \log\rho_{\Tt_m,\tau_m} \varphi \d\x\d t \underset{m\to\infty}\longrightarrow 
\iint_{Q_T}\rhoe\ell \varphi \d\x\d t, \qquad \forall \varphi \in C^\infty_c(Q_T).
\ee
Let us now identify $\ell$ as $\log(\rhoe)$ thanks to Minty's trick. Let $\k>0$ and $\varphi \in C^\infty_c(Q_T; \R_+)$ be arbitrary, then 
thanks to~\eqref{eq:Moussa}, 
\[
0 \leq \iint_{Q_T}\left(\rho_{\Tt_m,\tau_m} - \k\right)\left(\log\rho_{\Tt_m,\tau_m} - \log\k\right)\varphi \d\x \d t
\underset{m\to\infty}\longrightarrow  \iint_{Q_T}\left(\rhoe - \k\right)\left(\ell - \log\k\right)\varphi \d\x \d t.
\]
As a consequence, $\left(\rhoe - \k\right)\left(\ell - \log\k\right) \geq 0$ a.e. in $Q_T$ for all $\k>0$, which holds if and only if $\ell = \log\rhoe$.
To finalize the proof of~\eqref{eq:conv_rho_Tt}--\eqref{eq:conv_log_Tt}, define 
\[
c_m = (\rho_{\Tt_m,\tau_m} - \rhoe)(\log\rho_{\Tt_m,\tau_m} - \log\rhoe) \in L^1(Q_T;\R_+), \qquad \forall m \geq 1. 
\]
Then \eqref{eq:Moussa} implies that 
\[
\iint_{Q_T}c_m \varphi \d\x \d t \underset{m\to\infty}\longrightarrow 0, \qquad  \forall \varphi \in C^\infty_c(Q_T),\;\varphi\geq 0.
\]
As a consequence, $c_m$ tends to $0$ almost everywhere in $Q_T$, which implies the $\rho_{\Tt_m,\tau_m}$ 
tends almost everywhere towards $\rhoe$ (up to a subsequence). 
Then~\eqref{eq:conv_rho_Tt}--\eqref{eq:conv_log_Tt} follow from Vitali's convergence theorem (see for instance~\cite[Chap. XI, Theorem 3.9]{Visintin96}).

Finally, one has $\phi_{\Tt,\tau} = \log\rho_{\Tt,\tau} + V_\Tt - r_{\Tt,\tau}$. In view of the above discussion, 
the right-hand side converges strongly in $L^1(Q_T)$ up to a subsequence towards $\log\rhoe + V$, then so does the left-hand side. 
This provides~\eqref{eq:conv_phi_Tt} and concludes the proof of Proposition~\ref{prop:compact}.
\end{proof}

Next lemma shows that $\rho_{\Sig,\tau}$ shares the same limit $\rhoe$ as $\rho_{\Tt,\tau}$. 
\begin{lem}\label{lem:conv_rho_Sig}
Assume that  $\rhoe^0 \geq \rho_\star \in (0,+\infty)$, then 
\[
\left\|\rho_{\Sig_m,\tau_m} - \rho_{\Tt_m,\tau_m}\right\|_{L^1(Q_T)} \underset{m\to\infty}\longrightarrow 0. 
\]
\end{lem}
\begin{proof}
Thanks to Lemma~\ref{lem:rho_Sig}, it follows from the de La Vallée-Poussin and Dunford Pettis theorems that 
$\left(\rho_{\Sig_m,\tau_m}\right)_{m\geq1}$ is relatively compact for the weak topology of $L^1(Q_T)$.
Combining this with~\eqref{eq:L1_weak}, we infer that, up to a subsequence, $\left(\rho_{\Sig_m,\tau_m} - \rho_{\Tt_m,\tau_m}\right)_{m\geq1}$ 
converges towards some $w$ weakly in $L^1(Q_T)$. Thanks to Vitali's convergence theorem, it suffices to show that 
from any subsequence of $\left(\rho_{\Sig_m,\tau_m} - \rho_{\Tt_m,\tau_m}\right)_{m\geq1}$, one can extract a subsequence that
 tends to $0$ a.e. in $Q_T$ (so that the whole sequence converges towards $w=0$), or equivalently 
 \be\label{eq:poipoi}
 \left\|\log\rho_{\Sig_m,\tau_m} - \log\rho_{\Tt_m,\tau_m}\right\|_{L^1(Q_T)} \underset{m\to\infty}\longrightarrow 0, 
 \ee
 since both 
 $\left(\rho_{\Sig_m,\tau_m}\right)_{m\geq 1}$ and 
$\left(\rho_{\Tt_m,\tau_m}\right)_{m\geq 1}$ are bounded away from $0$ thanks to Lemma~\ref{lem:inf_phi}.
Bearing in mind the definition~\eqref{eq:rho_Sig_def} of $\rho_{\Sig_m,\tau_m}$, and 
one has 
\[
 \left\|\log\rho_{\Sig,\tau} - \log\rho_{\Tt,\tau}\right\|_{L^1(Q_T)} 
 \leq  
 \sum_{n=1}^N \tau  \sum_{\sig=K|L \in \Sig}m_{\Delta_\sig} |\log\rho_K^n - \log\rho_L^n|.
 \]
 Using~\eqref{eq:FP_HJ} and the triangle inequality, one gets that 
 \[
  \left\|\log\rho_{\Sig,\tau} - \log\rho_{\Tt,\tau}\right\|_{L^1(Q_T)}  \leq R_1 + R_2 + T R_3,
 \]
 with 
\[
 R_1 = \sum_{n=1}^N \tau  \sum_{\sig=K|L \in \Sig}m_{\Delta_\sig} |\phi_K^n - \phi_L^n|, \qquad
 R_2 = \sum_{n=1}^N \tau  \sum_{\sig=K|L \in \Sig}m_{\Delta_\sig} |r_K^n - r_L^n|, \]
 and 
 \[
 R_3 =  \sum_{\sig=K|L \in \Sig}m_{\Delta_\sig} |V_K - V_L|.
 \]
 Using again that $d m_{\Delta_\sig} = d_\sig m_\sig \leq \zeta h_\Tt m_\sig$ thanks to~\eqref{eq:reg_1}, one has 
 \[
 R_1 \leq \frac{\zeta}d h_\Tt  \sum_{n=1}^N \tau  \sum_{\sig=K|L \in \Sig}m_{\sig} |\phi_K^n - \phi_L^n| \leq \frac{\cter{cte:BV} \zeta}d h_\Tt 
  \underset{m\to\infty}\longrightarrow 0.
 \]
 Since $|r_K^n - r_L^n| \leq r_K^n + r_L^n$, the regularity assumption~\eqref{eq:reg_3} on the mesh implies that 
 \[
 R_2 \leq \sum_{n=1}^N \tau \sum_{K\in\Tt} \sum_{\sig \in \Sig_K} m_{\Delta_\sig} r_K^n \leq \zeta \|r_{\Tt,\tau}\|_{L^1(Q_T)}
 \underset{m\to\infty}\longrightarrow 0.
 \]
Since $V$ is Lipschitz continuous, $ |V_K - V_L| \leq \|\grad V\|_\infty d_\sig \leq \zeta \|\grad V\|_\infty h_\Tt$ for all $\sig = K|L \in\Sig$
thanks to~\eqref{eq:reg_1}. Therefore, 
\[
R_3 \leq  \zeta \|\grad V\|_\infty |\O| h_\Tt  \underset{m\to\infty}\longrightarrow 0, 
\]
so that~\eqref{eq:poipoi} holds, concluding the proof of Lemma~\ref{lem:conv_rho_Sig}.
\end{proof}

\subsection{Convergence towards the unique weak solution}\label{ssec:identify}

Our next lemma is a important step towards the identification of the limit $\rhoe$ as a weak solution to the continuous 
Fokker-Planck equation~\eqref{eq:FokkerPlanck}. Define the vector field $\bF_{\Sig,\tau}: Q_T \to \R^d$ by 
\[
\bF_{\Sig,\tau}(\x,t) = \begin{cases}
d \rho_\sig^n\frac{\phi_K^n - \phi_L^n}{d_\sig} \n_{K\sig} & \text{if}\;(\x,t) \in \Delta_\sig\times(t^{n-1},t^n], \\
0 & \text{otherwise}.
\end{cases}
\]
\begin{lem}\label{prop:flux}
Assume that $\rhoe^0 \geq \rho_\star \in (0,+\infty)$, then, up to a subsequence, 
the vector field $\bF_{\Sig_m,\tau_m}$ converges weakly in $L^1(Q_T)^d$ towards $-\grad \rhoe - \rhoe \grad V$ as 
$m$ tends to $+\infty$. Moreover, $\sqrt\rhoe$ belongs to $L^2((0,T);H^1(\O))$, while $\rhoe$ belongs to $L^2((0,T);W^{1,1}(\O))$.
\end{lem}
\begin{proof}
Let us introduce the inflated discrete gradient $\bG_{\Sig,\tau}$ of $\phi_{\Tt,\tau}$ defined by
\[
\bG_{\Sig,\tau}(\x,t) = \begin{cases}
d \frac{\phi_L^n - \phi_K^n}{d_\sig} \n_{K\sig} & \text{if}\;(\x,t) \in \Delta_\sig\times(t^{n-1},t^n], \\
0 & \text{otherwise}, 
\end{cases}
\]
so that $\bF_{\Sig,\tau} = - \rho_{\Sig,\tau}\bG_{\Sig,\tau}$. Thanks to Lemma~\ref{lem:L2H1_phi}, 
\[
\|\bG_{\Sig,\tau}\|_{L^2(Q_T)^d}^2 = d \sum_{n=1}^N \tau \sum_{\sig=K|L\in\Sig} a_\sig (\phi_K^n - \phi_L^n)^2 \leq d \cter{cte:L2H1_phi_2}, 
\]
thus we know that, up to a subsequence, $\bG_{\Sig,\tau}$ converges weakly towards some $\bG$ in $L^2(Q_T)^d$ as $m$ tends to $+\infty$.
Since $\phi_{\Tt,\tau}$ tends to $\log\rhoe + V$, cf.~\eqref{eq:conv_phi_Tt}, then the weak consistency of the inflated gradient~\cite{CHLP03,EG03} implies that 
$\bG = \grad (\log\rhoe+V)$. 

Define now $\bH_{\Sig,\tau}=\sqrt{\rho_{\Sig,\tau}}\bG_{\Sig,\tau}$, then using again Lemma~\ref{lem:L2H1_phi}, 
\[
\|\bH_{\Sig,\tau}\|_{L^2(Q_T)^d}^2 = d \sum_{n=1}^N \tau \sum_{\sig=K|L\in\Sig} a_\sig \rho_\sig^n (\phi_K^n - \phi_L^n)^2 \leq d \cter{cte:L2H1_phi_1}, 
\]
so that there exists $\bH \in L^2(Q_T)^d$ such that, up to a subsequence, $\bH_{\Sig,\tau}$ tends to $\bH$ weakly in $L^2(Q_T)^d$.
But since $\sqrt{\rho_{\Sig,\tau}}$ converges strongly towards $\sqrt{\rhoe}$ in $L^2(Q_T)$, cf. Lemma~\ref{lem:rho_Sig}, 
and since $\bG_{\Sig,\tau}$ tends weakly towards $\grad (\log\rhoe +V)$ in $L^2(Q_T)^d$, we deduce that 
$\bH_{\Sig,\tau}$ tends weakly in $L^1(Q_T)^d$ towards $\sqrt{\rhoe} \grad (\log\rhoe +V) = 2 \grad \sqrt{\rhoe} + \sqrt{\rhoe}\grad V =\bH$.
In particular, $\sqrt{\rhoe}$ belongs to $L^2((0,T);H^1(\O))$. 
Now, we can pass in the limit $m\to+\infty$ in $\bF_{\Sig,\tau} = - \sqrt{\rho_{\Sig,\tau}}\bH_{\Sig,\tau}$, leading to the desired result. 
\end{proof}

In order to conclude the proof of Theorem~\ref{thm:conv}, it remains to check that 
any limit value $\rhoe$ of the scheme is a solution to the Fokker-Planck equation~\eqref{eq:FokkerPlanck} in the distributional sense. 

\begin{prop}\label{prop:identify}
Let $\rhoe$ be a limit value of $\left(\rho_{\Tt_m,\tau_m}\right)_{m\geq 1}$ as described in Section~\ref{ssec:compact}, then for all 
$\varphi\in C^\infty_c(\ov\O\times[0,T))$, one has 
\be\label{eq:weak_for}
\iint_{Q_T}\rhoe\p_t\varphi\d\x \d t + \int_\O \rhoe^0 \varphi(\cdot,0)\d\x - \iint_{Q_T}(\rhoe \grad V + \grad \rhoe)\cdot \grad \varphi\d\x\d t = 0. 
\ee
\end{prop}
\begin{proof}
Given $\varphi\in C^\infty_c(\ov\O\times[0,T))$, we denote by $\varphi_K^n = \varphi(\x_K,t^n)$. Then multipying~\eqref{eq:FP_cons} by $-\varphi_K^{n-1}$ 
and summing over $K\in\Tt$ and $n\in\{1,\dots, N\}$ leads to 
\[
B_1 + B_2 + B_3= 0, 
\]
where we have set 
\[
B_1 = \sum_{n=1}^N \tau \sum_{K\in\Tt} m_K \frac{\varphi_K^n - \varphi_K^{n-1}}{\tau} \rho_K^n,  \qquad 
B_2 = \sum_{K\in\Tt} m_K \varphi_K^0 \rho_K^0, 
\]
and 
\[
B_3 = - \sum_{n=1}^N \tau \sum_{\sig = K|L \in \Sig} a_\sig \rho_\sig^n \left(\phi_K^n - \phi_L^n\right) \left( \varphi_K^{n-1} - \varphi_{L}^{n-1} \right).
\]
Since $\rho_{\Tt,\tau}$ converges in $L^1(Q_T)$ towards $\rhoe$, cf. Proposition~\ref{prop:compact}, and since $\varphi$ is smooth,
\[
B_1 \underset{m\to\infty}\longrightarrow \iint_{Q_T}\rhoe\p_t\varphi\d\x \d t.
\]
It follows from the definition~\eqref{eq:rhoK0} of $\rho_K^0$ that the piecewise constant function 
$\rho_\Tt^0$, defined by $\rho_\Tt^0(\x)=\rho_K^0$ if $\x\in\Tt$, converges in $L^1(\O)$ towards $\rhoe^0$. 
Therefore, since $\varphi$ is smooth, 
\[
B_2 \underset{m\to\infty}\longrightarrow \int_{\O}\rhoe^0\varphi(\cdot,0)\d\x.
\]
Let us define 
\[
B_3' = \iint_{Q_T} \bF_{\Sig,\tau}\cdot \grad \varphi \d\x\d t. 
\]
Then it follows from Lemma~\ref{prop:flux} that 
\[
B_3' \underset{m\to\infty}\longrightarrow - \iint_{Q_T}(\rhoe \grad V + \grad \rhoe)\cdot \grad \varphi\d\x\d t.
\]
To conclude the proof of Proposition~\ref{prop:identify}, it only remains to check that 
\[
\left|B_3 - B_3'\right| \leq \sum_{n=1}^N \tau \sum_{\sig = K|L \in \Sig} a_\sig \rho_\sig^n \left|\phi_K^n - \phi_L^n\right| 
 \left|  \varphi_K^{n-1} - \varphi_L^{n-1} + 
\frac{1}{\tau m_{\Delta_\sig}} \int_{t^{n-1}}^{t^n}\int_{\Delta_\sig} d_\sig \grad \varphi \cdot \n_{KL} \right|\d\x\d t.
\]
Since $\varphi$ is smooth and since $d_\sig \n_{KL} = \x_K - \x_L$ thanks to the orthogonality condition satisfied by the mesh, 
\[
 \left|  \varphi_K^{n-1} - \varphi_L^{n-1} + 
\frac{1}{\tau m_{\Delta_\sig}} \int_{t^{n-1}}^{t^n}\int_{\Delta_\sig} d_\sig \grad \varphi \cdot \n_{KL} \right|\d\x\d t
 \leq C_\varphi d_\sig (\tau+d_\sig)
\]
for some $C_\varphi$ depending only on $\varphi$. 
Therefore, 
\[
\left|B_3 - B_3'\right| \leq C_\varphi  (\tau+d_\sig)  \sum_{n=1}^N \tau \sum_{\sig = K|L \in \Sig} m_\sig \rho_\sig^n \left|\phi_K^n - \phi_L^n\right|.
\]
Applying Cauchy-Schwarz inequality, one gets that 
\[
\left|B_3 - B_3'\right| \leq C_\varphi  (\tau+d_\sig) \cter{cte:L2H1_phi_1} d \left\| \rho_{\Sig,\tau} \right\|_{L^1(Q_T)} \underset{m\to\infty}\longrightarrow 0
\]
thanks to Lemma~\ref{lem:rho_Sig}.
\end{proof}


\section{Numerical results}\label{sec:num}

To check the correctness and reliability of our formulation we performed some numerical tests. Before that, we are going to present some details on the solution of the nonlinear system involved in the scheme. 

\subsection{Newton method}

Due to the explicit formulation of the optimality condition of the saddle point problem (\ref{eq:LJKO_dual}), it appears extremely convient to use a Newton method for their solution. 
Given $\bu^{n-1}=(\bphi^{n-1},\brho^{n-1}) \in \R^{2\Tt}$ solution of the scheme at the time step $n-1$, the Newton method aims at constructing a sequence of approximations of $\bu^n$ as $\bu^{n,k+1} = \bu^{n,k} + \boldsymbol{d}^k$, $\boldsymbol{d}^k = (\boldsymbol{d}_{\bphi}^k,\boldsymbol{d}_{\brho}^k)$ being the Newton direction, solution to the block-structured system of equations
\begin{equation}\label{eq:Newton}
\bJ^{k}  \boldsymbol{d}^k = 
\begin{bmatrix}
\bJ_{\bphi,\bphi}^{k} & \bJ_{\bphi,\brho}^{k}  \\
\bJ_{\brho,\bphi}^{k} & \bJ_{\brho,\brho}^{k} \\
\end{bmatrix}
\begin{bmatrix}
\boldsymbol{d}_{\bphi}^k \\ 
\boldsymbol{d}_{\brho}^k \\
\end{bmatrix}
= \begin{bmatrix}
\bff_{\bphi}^k \\ 
\bff_{\brho}^k \\
\end{bmatrix}.
\end{equation}
In the above linear system,
$\bff_{\bphi}^k$ and $\bff_{\brho}^k$ are the discrete HJ and continuity equations evaluated in $\bu^{n,k}$, and 
$\bJ_{\bphi,\bphi}^k$, $\bJ_{\bphi,\brho}^k$, $\bJ_{\brho,\bphi}^k$ and $\bJ_{\brho,\brho}^k$ are the four blocks of the Hessian matrix $\bJ^k$ of the discrete functional in (\ref{eq:LJKO_dual}) evaluated in $\bu^{n,k}$. The sequence converges to the unique solution $\bu^n$ as soon as the initial guess is sufficiently close to it, which is ensured for a sufficiently small time step by taking $\bu_0^n=\bu^{n-1}$.
The algorithm stops when the $\ell^{\infty}$ norm of the discrete equations is smaller than a prescribed tolerance or if the maximum number of iterations is reached. It is possible to implement an adaptative time stepping: if the Newton method converges in few iterations the time step $\tau$ increases; if it reaches the maximum number of iterations the time step is decreased and the method restarted. Issues could arise if the iterate $\bu^{n,k}$ reaches negative values, especially if the energy is not defined for negative densities. To avoid this problem two possible strategies may be implemented: the iterate may be projected on the set of positive measure by taking $\bu^{n,k} = (\bu^{n,k})^+$; the method may be restarted with a smaller time step.

In case of a local energy functional, as it is the case for the Fokker-Planck and many more examples, the block $\bJ_{\brho,\brho}^k$ is diagonal and therefore straightforward to invert. System (\ref{eq:Newton}) can be rewritten in term of the Schur complement and solved for $\boldsymbol{d}_{\bphi}^k$ as
\begin{equation}\label{eq:schur}
\begin{bmatrix}
\bJ_{\bphi,\bphi}^k -    \bJ_{\bphi,\brho}^k\, (\bJ_{\brho,\brho}^k)^{-1} \,\bJ_{\brho,\bphi}^k
\end{bmatrix}
\boldsymbol{d}_{\bphi}^k = \bff_{\bphi}^k - \bJ_{\bphi,\brho}^k \, (\bJ_{\brho,\brho}^k)^{-1} \, \bff_{\brho}^k,
\end{equation}
while $\boldsymbol{d}_{\brho}^k = (\bJ_{\brho,\brho}^k)^{-1} \, (\bff_{\brho}^k - \bJ_{\brho,\bphi}^k \, \boldsymbol{d}_{\bphi}^k )$.
\begin{prop}
The Schur complement $\bS^k = \bJ_{\bphi,\bphi}^k -  \bJ_{\bphi,\brho}^k \, (\bJ_{\brho,\brho}^k)^{-1} \, \bJ_{\brho,\bphi}^k$ is symmetric and negative definite.
\end{prop}
\begin{proof}
$\bS^k$ is symmetric since $\bJ_{\bphi,\bphi}^k$ and $\bJ_{\brho,\brho}^k$ are, while $\bJ_{\bphi,\brho}^k=(\bJ_{\brho,\bphi}^k)^T$. The matrix 
$\bJ_{\brho,\brho}^k$ is positive definite since the problem is strictly convex, whereas $\bJ_{\bphi,\bphi}^k$ is negative definite 
if $\rho_K^{n,k} > 0 , \forall K \in \Tt$, since the problem is strictly concave, but it is semi-negative definite if the density vanishes somewhere. Therefore, it is sufficient to show that the matrix  $\bJ_{\bphi,\brho}^k = (\bJ_{\brho,\bphi}^k)^T= \bM + \bA^k$ is invertible. $\bM$ is a diagonal matrix such that $(\bM)_{K,K}=m_K$, whereas
\[
(\bA^k)_{K,K} = \tau \sum_{\sigma = K|L \in \Sigma_K} a_{\sigma} (\phi_K^{n,k}-\phi_L^{n,k})^+ \geq 0,
\]
and, for $L\neq K$,
\[
(\bA^k)_{K,L} =
-\tau a_{\sigma} (\phi_L^{n,k}-\phi_K^n)^+ \leq 0 \quad \text{if}\; \sig = K|L, \qquad (\bA^k)_{K,L} = 0 \quad \text{otherwise}.
\]
Therefore the columns of $\bA^k$ sum up to $0$, so that  $(\bJ_{\bphi,\brho}^k)$ is a column M-matrix \cite{Fuhrmann01} and thus invertible.
\end{proof}
In case the matrix $\bJ_{\brho,\brho}^k$ is simple to invert it is then possible to decrease the computational complexity of the solution of system \eqref{eq:Newton}. Moreover, it is possible to exploit for the solution of system \eqref{eq:schur} symmetric solvers, which are computationally more efficient than non-symmetric ones.

\subsection{Fokker-Planck equation}

We first tackle the gradient flow of the Fokker-Planck energy, namely \cref{eq:FokkerPlanck}. In \cref{sec:conv} we showed the $L^1$ convergence of the scheme. Consider the specific potential $\rho V(\mathbf{x})=-\rho gx$: for this case it is possible to design an analytical solution and test the convergence of the scheme. Consider the domain $\Omega=[0,1]^2$, the time interval $[0,0.25]$ and the following analytical solution of the Fokker-Planck equation (built from a one-dimensional one):
\[
\rho(x,y,t) = \exp(-\alpha t+\frac{g}{2}x)(\pi \cos(\pi x)+\frac{g}{2}sin(\pi x))+\pi \exp(g(x-\frac{1}{2})),
\]
where $\alpha=\pi^2+\frac{g^2}{4}$. On the domain $\Omega=[0,1]^2$, the function $\rho(x,y,t)$ is positive and satisfies the mixed boundary conditions $(\nabla \rho + \rho \nabla V) \cdot \n |_{\partial\Omega} = 0$. We want to exploit the knowledge of this exact solution to compute the error we commit in the spatial and time integration. Consider a sequence of meshes $\left(\Tt_m, \ov \Sigma_m, {(\x_K)}_{K\in\Tt_m}\right)$ with decreasing mesh size $h_{\mathcal{T}_m}$ and a sequence of decreasing time steps $\tau_m$ such that $\frac{h_{\mathcal{T}_{m+1}}}{h_{\mathcal{T}_m}}=\frac{\tau_{m+1}}{\tau_m}$. In particular, we used a sequence of Delaunay triangular meshes such that the mesh size halves at each step, obtained subdividing at each step each triangle into four using the edges midpoints. Three subsequent partitioning of the domain are shown in figure \ref{fig:mesh}.
\begin{figure}
	\centering
	\subfloat{\includegraphics[width=0.33\textwidth]{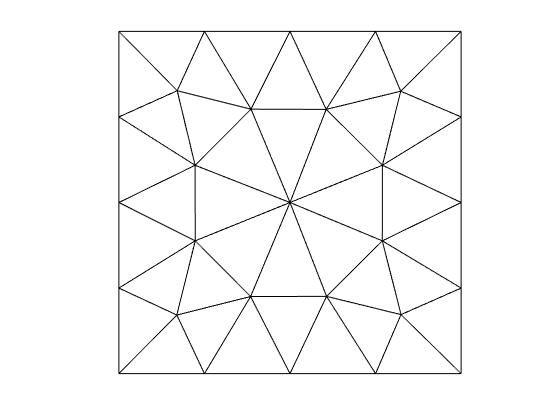}}
	\subfloat{\includegraphics[width=0.33\textwidth]{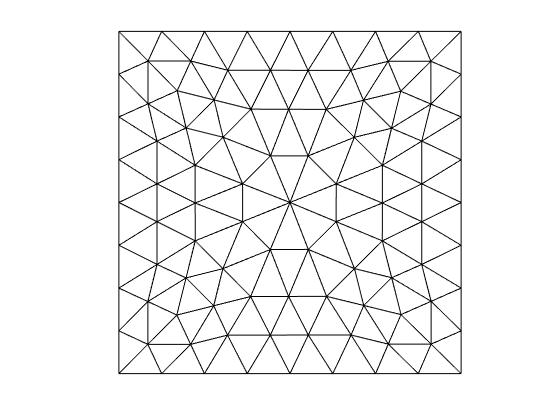}}
	\subfloat{\includegraphics[width=0.33\textwidth]{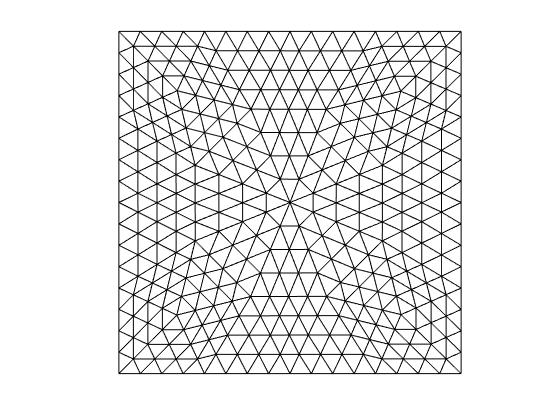}}
	\caption{Sequence of regular triangular meshes.}
	\label{fig:mesh}
\end{figure}
Let us introduce the following mesh-dependent errors:
\[
\begin{gathered}
\epsilon_1^n = \sum_{K \in \Tt_m} |\rho_K^n-\rho(\x_K,n \tau)| m_K, \quad \rightarrow \quad \text{discrete $L^1$ error}\\
\epsilon_{L^{\infty}} = \max_{n} (\epsilon^1_n), \quad \rightarrow \quad \text{discrete $L^{\infty}((0,T);L^1(\Omega))$ error}, \\
\epsilon_{L^1} = \sum_{n} \tau \, \epsilon_1^n, \quad \rightarrow \quad \text{discrete $L^1((0,T);L^1(\Omega))$ error},
\end{gathered}
\]
where $\rho(\x_K,n \tau_m)$ is the value in the cell center of the triangle $K$ of the analytical solution at time $n \tau_m$, $n$ running from 0 to the total number of time steps $N_m$. The upstream Finite Volume scheme with implicit Euler discretization of the temporal derivative is known to exhibit order one of convergence applied to this problem, both in time and space. This means that the $L^{\infty}((0,T);L^1(\Omega))$ and $L^1((0,T);L^1(\Omega))$ errors halve whenever $h_{\Tt}$ and $\tau$ halve. We want to inspect whether scheme (\ref{LJKOh}) recovers the same behavior.

For the sequence of meshes and time steps, for $m$ going from one to the total number of meshes, we computed the solution to the linear Fokker-Planck equations and the errors, using both the Finite Volume scheme and scheme (\ref{LJKOh}). The results are shown in Table \ref{tab:errors00}. For each mesh size and time step $m$, it is represented the error together with the rate with respect to the previous one. Scheme (\ref{LJKOh}) exhibits the same order of convergence of the FV scheme. It is noticeable that the rate of convergence of scheme (\ref{LJKOh}) senses a big drop and then recovers order one, especially in the $L^{\infty}((0,T);L^1(\Omega))$ error. This is due to the fact that the initial condition $\rho(\x_K,0)$ is too close to zero, and in particular equal to zero on the set ${1}\times [0,1]$, and scheme (\ref{LJKOh}) tends to be repulsed away from zero due to the singularity of the gradient of the first variation of the energy. In Table \ref{tab:errors05} we repeated the convergence test for the time interval $[0.05,0.25]$: the convergence profile sensibly improves.
\begin{table}
	\centering
	\caption{Time-space convergence for the two schemes. Integration on the time step $[0,0.25].$}
	\label{tab:errors00}
	\begin{tabular}{cccccccccc}
		\toprule
		& & \multicolumn{4}{c}{FV} & \multicolumn{4}{c}{LJKO} \\
		\hline
		h & dt & $\epsilon_{L^{\infty}}$ & $r$ & $\epsilon_{L^1}$ & $r$ & $\epsilon_{L^{\infty}}$ & $r$ & $\epsilon_{L^1}$ & $r$ \\
		\midrule
		0.2986 &   0.0500 &   0.1634 &  /  &  0.0350   & /   &  0.1463    &     /  &  0.0334   &      / \\
		0.1493  &  0.0250 &  0.0856  &  0.932   &  0.0176  &  0.997   & 0.0651  & 1.169  &  0.0145  &  1.120 \\
		0.0747  &  0.0125 &   0.0434 &   0.979 &   0.0087   & 1.015   &  0.0449 &   0.535 &   0.0066  &  1.134 \\
		0.0373  &  0.0063 &   0.0218 &   0.996 &   0.0043   & 1.009 &  0.0297 &   0.598  &  0.0033 &   1.007 \\
		0.0187  &  0.0031 &   0.0109  &  0.999  & 0.0022  &  1.004  &  0.0174 &   0.770  &  0.0017 &   0.943 \\
		0.0093 & 0.0016  &  0.0054 &   1.000  &  0.0011  &  1.001  & 0.0095  &  0.870  &  0.0009  &  0.947 \\
		\bottomrule
	\end{tabular}
\end{table}
\begin{table}
	\centering
	\caption{Time-space convergence for scheme (\ref{LJKOh}). Integration on the time step $[0.5,0.25].$}
	\label{tab:errors05}
	\begin{tabular}{cccccc}
		\toprule
		& & \multicolumn{4}{c}{LJKO} \\
		\hline
		h & dt & $\epsilon_{L^{\infty}}$ & $r$ & $\epsilon_{L^1}$ & $r$ \\
		\midrule
		0.2986  &  0.0500 &   0.1186  &       /   & 0.0216    &     / \\
    		0.1493  &  0.0250 &   0.0618  &  0.9411  &  0.0109  &  0.9857 \\
   		0.0747  &  0.0125 &   0.0307  &  1.0110  &  0.0053  &  1.0311 \\
    		0.0373  &  0.0063 &   0.0152  &  1.0116  &  0.0026  &  1.0213 \\
    		0.0187  &  0.0031 &   0.0076  &  1.0078 &   0.0013  &  1.0119 \\
   		0.0093  &  0.0016 &   0.0038  &  1.0042 &   0.0006  &  1.0062 \\
		\bottomrule
	\end{tabular}
\end{table}

To further investigate and compare the behavior of scheme (\ref{LJKOh}) with the FV, we computed also the energy decay along the trajectory. We call dissipation the difference $\mathcal{E}(\rho)-\mathcal{E}(\rho^{\infty})$, where $\rho^{\infty}$ is the final equilibrium condition, the long time behavior. Since we are discretizing a gradient flow, its dissipation is a useful criteria to assess the goodness of the scheme.
The long time value of the energy is equal to:
\[
\begin{aligned}
\mathcal{E}(\lim_{t\rightarrow \infty} \rho) &=  \int_{\Omega} \lim_{t\rightarrow \infty} (\rho \log \rho - \rho g x) \d\x \\
&= \exp(\frac{g}{2})(\frac{\pi\log(\pi)}{g}+\frac{\pi}{2}-\frac{\pi}{g})+\exp(-\frac{g}{2})(-\frac{\pi \log(\pi)}{g}-\frac{\pi}{2}+\frac{\pi}{g}).
\end{aligned}
\]
It is possible to define the equilibrium solution also on the discrete dynamics on the grid. Namely, the equilibrium solution for the dynamic defined on the grid is
\[
\rho_K^{\infty} = M\exp(-V_K), \,\, V_K = V(\x_K),
\]
as it can be easily checked to be the unique minimizer of the discrete energy $\mathcal{E}_h= \sum_{K\in \mathcal{T}} E(\rho_K) m_K$ subject to the constraint of the conservation of the mass,
\[
\begin{gathered}
\frac{\partial }{\partial \rho_K} \big( \mathcal{E}_h+\lambda \sum_{K\in \mathcal{T}} (\rho_K -\rho_K^0)m_K \big)|_{\rho_K^{\infty}} = \big( \log \rho_K^{\infty} + 1 + V_K + \lambda \big) m_K = 0, \quad \forall K \in \mathcal{T} \\
\implies \quad \rho_K^{\infty} = \exp (-(1+\lambda)-V_K) = M\exp(-V_K), \quad \forall K \in \mathcal{T} ,
\end{gathered}
\]
with $\lambda$ lagrange multiplier associated with the constraint. $M$ is the constant that makes $(\rho_K^{\infty})_{K\in \mathcal{T}}$ have the same total mass:
\[
M = \frac{\sum_{K\in \mathcal{T}_h}\rho_K^0 m_K}{\sum_{K\in \mathcal{T}_h} \exp^{-V_K} m_K}.
\]
It is immediate to observe that this is indeed the equilibrium solution in the FV scheme, since with such density the potential in \eqref{eq:FV} is constant:
\[
\phi_K = \frac{\delta \mathcal{E}_h(\rho)}{\delta \rho_K}|_{\rho_K^{\infty}} = \log \rho_K^{\infty} +1 +V_K = \log M -V_K +1 +V_K = \log M +1, \quad \forall K \in \mathcal{T}.
\]
For the scheme (\ref{LJKOh}) instead, as it appears clear from \Cref{lem:Legendre}, whenever $\rho_K^n=\rho_K^{n-1}, \forall K\in \mathcal{T}$, as it is the case for an equilibrium solution, the potential is constant. From the potential equation one gets again
\[
\phi_K = \frac{\delta \mathcal{E}_h(\rho)}{\delta \rho_K}|_{\rho_K^{\infty}} = \log M +1, \forall K \in \mathcal{T}.
\]

In \Cref{fig:dissipation} it is represented the semilog plot of the dissipation of the system, computed for the scheme (\ref{LJKOh}) and the FV one, $\mathcal{E}^h(\rho_K)-\mathcal{E}^h(\rho_K^{\infty})$, and the real solution, $\mathcal{E}(\rho)-\mathcal{E}(\rho^{\infty})$.
In \Cref{fig:test1} it is noticeable that scheme (\ref{LJKOh}) scheme dissipates the energy faster than the FV one, being indeed a bit more diffusive. This is an expected behavior since the scheme is built to maximize the decrease of the energy and this is actually one of the main strength of the approach. 
In \Cref{fig:test2}, one can see that the two dissipations tends to the real one when finer mesh and smaller time step are used, for both schemes, despite the fact that (\ref{LJKOh}) still dissipates faster. In the end, in \Cref{fig:test3} it is remarkable that for a very small time step the dissipation of the two schemes tends to coincide, as it is expected. For the time parameter going to zero the two schemes coincide. 

\begin{figure}
	\centering
	\subfloat[][$T=3, \tau =0.01, h=0.1493$. \emph{Dissipation over the time interval $[0,T]$ and detail.}]
	{\includegraphics[width=0.5\textwidth]{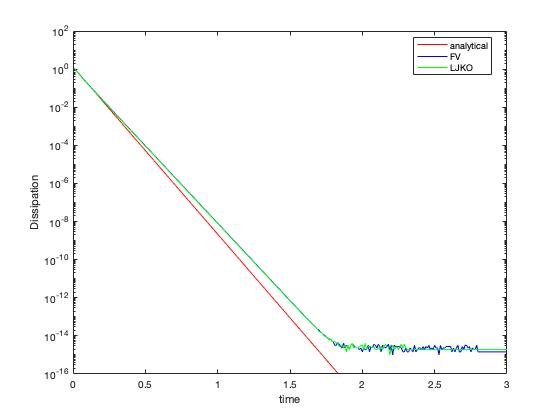}\includegraphics[width=0.5\textwidth]{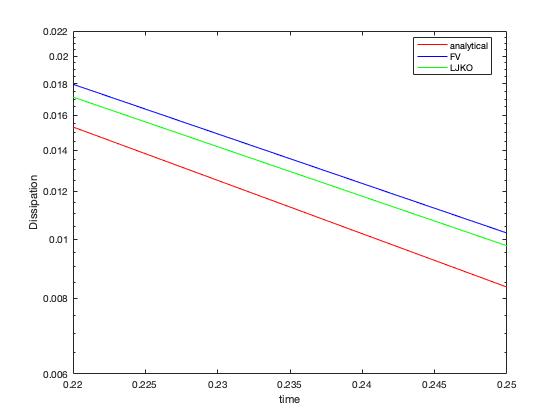}\label{fig:test1}} \\
	\subfloat[][$T=3, \tau =0.0063, h=0.0373$. \emph{Dissipation over the time interval $[0,T]$ and detail.}]
	{\includegraphics[width=0.5\textwidth]{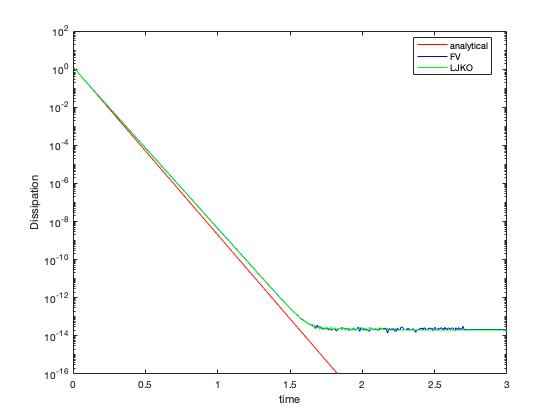}\includegraphics[width=0.5\textwidth]{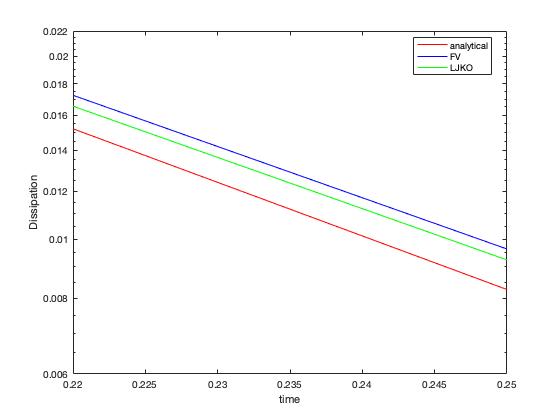}\label{fig:test2}} \\
	\subfloat[][$T=3, \tau =0.0001, h=0.1493$. \emph{Dissipation over the time interval $[0,T]$ and detail.}]
	{\includegraphics[width=0.5\textwidth]{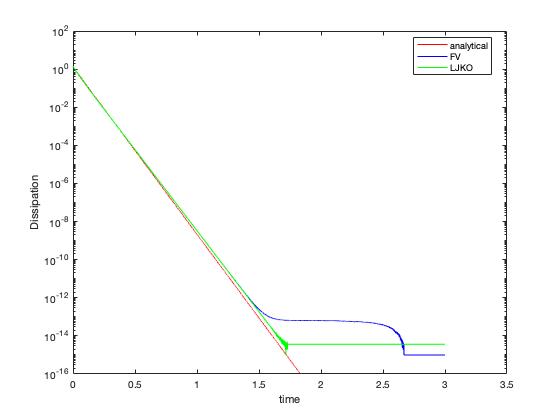}\includegraphics[width=0.5\textwidth]{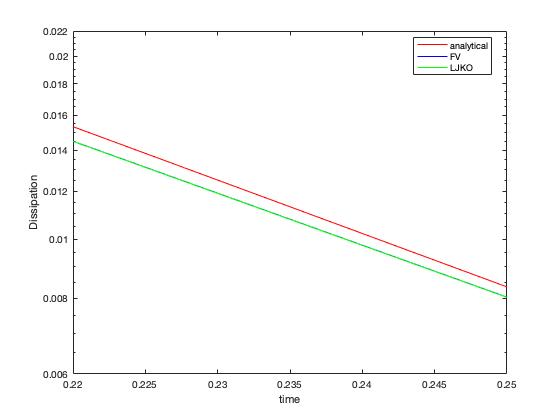}\label{fig:test3}}
	\caption{Comparison of the dissipation of the system computed with the two numerical schemes and in the real case. Semi-logarithmic plot.}
	\label{fig:dissipation}
\end{figure}

\subsection{Porous medium equation}

 Consider again the domain $\Omega=[0,1]^2$. The porous medium equation,
\[
\partial_t \rho = \Delta \rho^m + \div (\rho \grad V),
\]
has been proven in \cite{Otto01} to be a gradient flow in Wasserstein space with respect to the energy
\[
\mathcal{E}(\rho) = \int_{\O} \frac{1}{m-1} \rho^m + \rho V,
\]
for a given $m$ strictly greater than one. Our aim is to show that scheme (\ref{LJKOh}) works regardless of the positivity assumption on the initial measure $\rho_0$. For this reason, we use a confining potential $V(\x) = \frac{1}{2} || \x-0.5 ||^2_2$. The equilibrium solution of the gradient flow should then be the Barenblatt profile $\rho^{\infty}(\x) = \big( \frac{m-1}{2m} \max(1-||\x||^2) \big)^{\frac{1}{m-1}}$.

In \Cref{fig:porous_medium} the evolution of an initial density with compact support is shown for the case $m=4$. As expected, the solution converges towards the Barenblatt profile.

\begin{figure}
	\centering
	\graphicspath{{Figures/porous_medium/}}
	\subfloat[][t=0]{\includegraphics[width=0.5\textwidth]{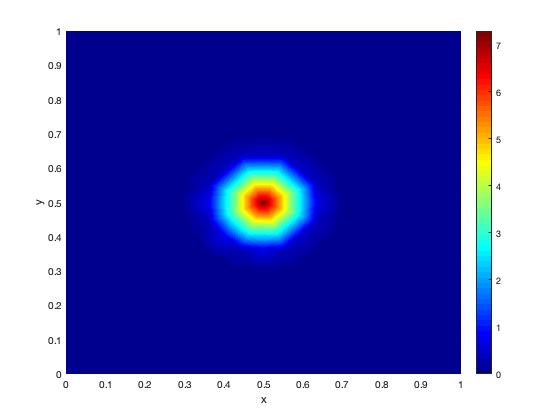}}
	\subfloat[][t=0.001]{\includegraphics[width=0.5\textwidth]{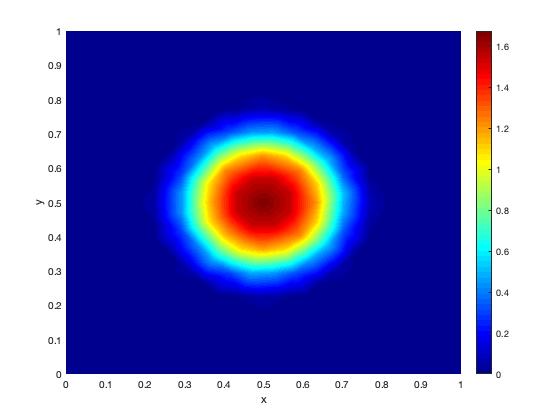}} \\
	\subfloat[][t=0.01]{\includegraphics[width=0.5\textwidth]{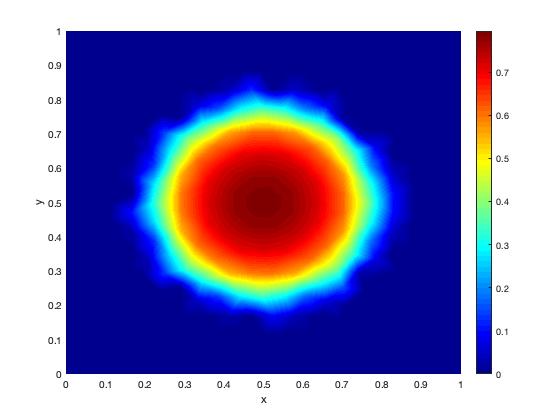}}
	\subfloat[][t=1]{\includegraphics[width=0.5\textwidth]{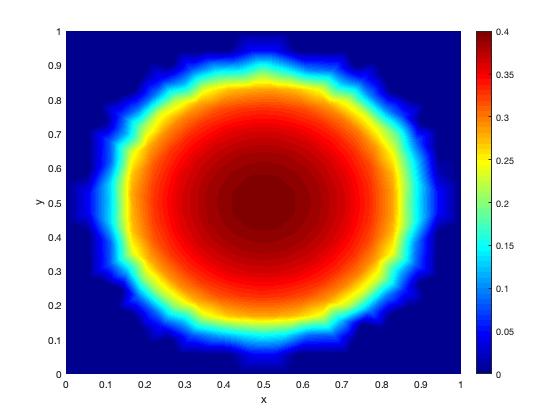}} \\
	\caption{Evolution of an initial density close to a dirac. In each picture the scaling is different for the sake of the representation.}
	\label{fig:porous_medium}
\end{figure}

\subsection{Salinity intrusion problem}

We want to show now that scheme (\ref{LJKOh}) can be used for the solutions of systems of equations of the type of \cref{eq:pde}.
We consider the problem of salinity intrusion in an unconfined aquifer. Under the assumption that the two fluids, the fresh and the salty water, are immiscible and the domains occupied by each fluid are separated by a sharp interface, the problem can be modeled via the system of equations
\be\label{eq:intrusion}
\begin{cases}
\partial_t f - \div (\nu f \grad(f+g+b)) = 0 \quad \text{in} \, \O \times (0,T) \\
\partial_t g - \div (g \grad(\nu f+g+b)) = 0 \quad \text{in} \, \O \times (0,T) \\
\end{cases}
\ee
completed with the no-flux boundary conditions
\[
\grad f \cdot \n = \grad g \cdot \n = 0 \quad \text{on} \,  \partial \O \times (0,T),
\]
and initial conditions $f(t=0) = f_0, g(t=0) = g_0$, with $f_0,g_0 \in L^{\infty}(\O), f_0,g_0 \ge 0$. The quantities $f$, $g$, and $b$ represent respectively the thickness of the 
fresh layer of water, the thickness of the salty water layer and the height of the bedrock. Therefore the quantity $b+g$ represents the height of the sharp interface 
separating the two fluids. The parameter $\nu = \frac{\rho_f}{\rho_s}$ is the ratio between the constant mass density of the fresh and salt water. 
\Cref{eq:intrusion} has been proven in \cite{LM13_Muskat} to be a Wasserstein gradient flow with respect to the energy
\[
\mathcal{E}(f,g) = \int_{\O} \Big( \frac{\nu}{2} (f+g+b)^2 + \frac{1-\nu}{2} (g+b)^2 \Big) \d \x.
\]

\begin{figure}
	\centering
	\subfloat[][t=0]{\includegraphics[width=0.5\textwidth]{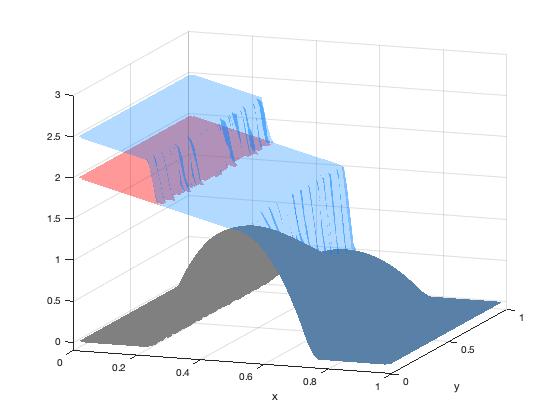}}
	\subfloat[][t=0.1]{\includegraphics[width=0.5\textwidth]{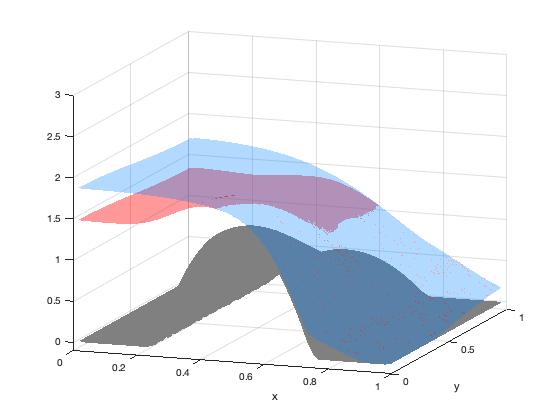}} \\
	\subfloat[][t=0.5]{\includegraphics[width=0.5\textwidth]{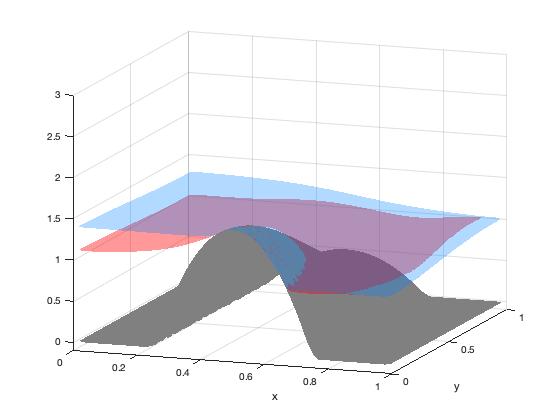}}
	\subfloat[][t=10]{\includegraphics[width=0.5\textwidth]{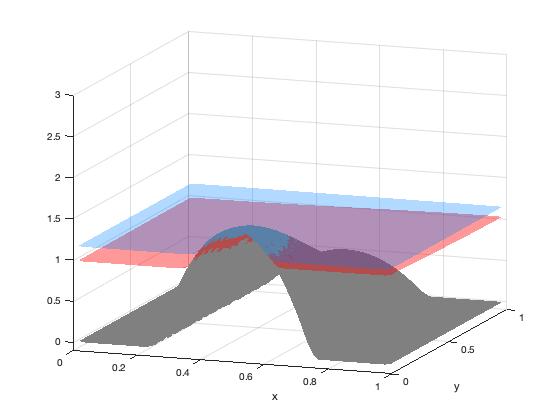}} \\
	\caption{Evolution of the two interfaces of salt (red) and fresh (blue) water.}
	\label{fig:salinity_intrusion}
\end{figure}
\Cref{fig:salinity_intrusion} represents the bedrock $b$ and the evolution of the surfaces of salt water, $b+g$, and of fresh water, $b+g+f$ (see~\cite{Ahmed_intrusion} for a full description of the test case).
Also this case is not covered from the theoretical analysis we performed on the convergence of the scheme but still scheme (\ref{LJKOh}) works. As already said, from numerical evidences the scheme works under much more general and mild hypotheses. \\


{\bf Acknowledgements.} 
CC acknowledges the support of the Labex CEMPI (ANR-11-LABX-0007-01).
GT acknowledges that this project has received funding
from the European Union’s Horizon 2020 research and innovation
programme under the Marie Skłodowska-Curie grant agreement No 754362.
We also thanks Guillaume Carlier and Quentin M\'erigot for fruitful discussions.
\begin{center}
\includegraphics[width=0.2\textwidth]{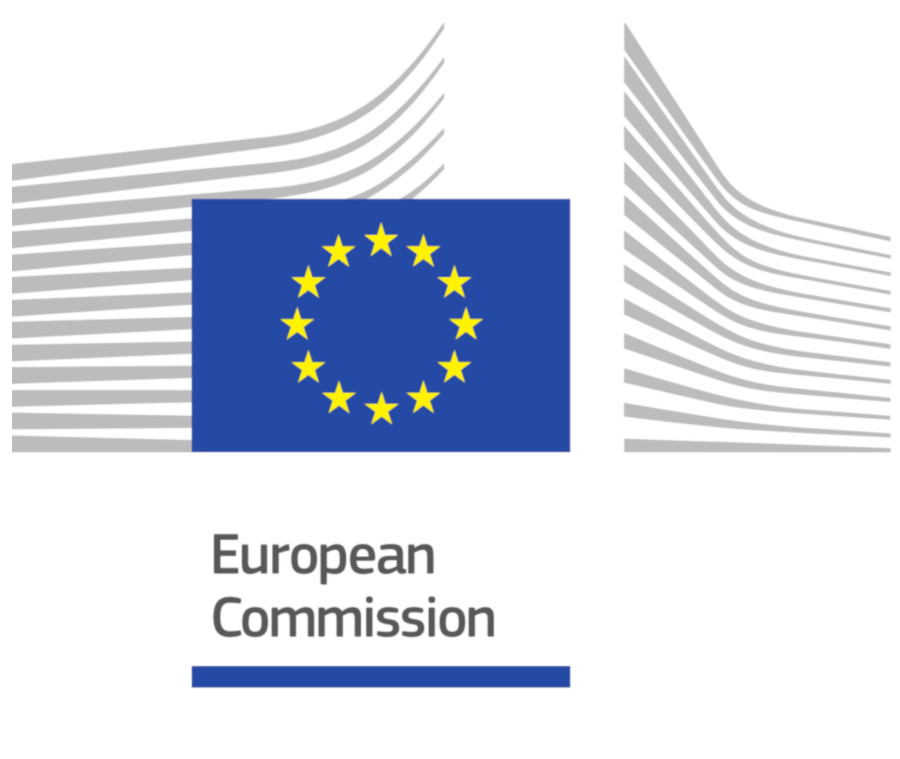}
\end{center}


\end{document}